\documentclass[11pt]{amsart}

\usepackage{eulervm}  

\usepackage[usenames,dvipsnames,svgnames,table]{xcolor} 
\usepackage{amssymb,amsfonts, amsrefs, eufrak}
\usepackage{amsmath,amsthm}

\usepackage{pgf, tikz} 
\usepackage[active]{srcltx}		
\usepackage{pdfsync}					
\usepackage{enumitem}
\usepackage[colorinlistoftodos]{todonotes}
\usepackage{dsfont}
\usepackage[pdfdisplaydoctitle,colorlinks,urlcolor=blue,linkcolor=blue,citecolor=blue]{hyperref}

\newcommand{\dd}{\,{\rm d}}
\newcommand{\ddd}{\,{\mathbf d}}

\newcommand{\e}{\epsilon}
\newcommand{\ee}{\varepsilon}

\renewcommand{\d}{\delta}

\newcommand{\la}{\lambda}

\renewcommand{\a}{\alpha}
\renewcommand{\b}{\beta}
\newcommand{\g}{\gamma}

\renewcommand{\k}{\kappa}
\renewcommand{\t}{\theta}
\newcommand{\p}{\partial}

\newcommand\R{{\mathbb{R}}}
\newcommand\1{{\mathds{1}}}
\renewcommand\P{{\mathcal{P}}}
\newcommand\Pp{{\mathbb{P}}}

\newcommand\A{{\mathcal{A}}}

\newcommand\N{{\mathbb{N}}}
\newcommand\M{{\mathfrak{M}}}

\newcommand\C{{\mathcal{C}}}
\newcommand{\La}{\Lambda}

\newcommand\Qq{{\mathbb{Q}}}
\newcommand\J{{\mathcal{J}}}
\newcommand\D{{\mathcal{D}}}
\renewcommand\S{{\mathcal{S}}}

\newcommand\E{{\mathbb{E}}}
\newcommand\T{{\mathbb{T}}}
\newcommand\F{{\mathcal{F}}}

\newcommand\X{{\mathbb X}}

\newcommand\Z{{\mathcal{Z}}}

\newcommand{\bb}[1]{{\mathbb #1}}

\newtheorem{theorem}{Theorem}[section]
\newtheorem{proposition}[theorem]{Proposition}
\newtheorem{lemma}[theorem]{Lemma}

\theoremstyle{definition}
\newtheorem{definition}[theorem]{Definition}

\theoremstyle{remark}
\newtheorem{remark}[theorem]{Remark}
\numberwithin{equation}{section}

\begin{document}

\title[On Mean Field Games models for exhaustible commodities trade]
{On Mean Field Games models for exhaustible commodities trade}

\author{P. Jameson Graber}
\address{J. Graber: Baylor University, Department of Mathematics;
One Bear Place \#97328;
Waco, TX 76798-7328 \\
Tel.: +1-254-710- \\
Fax: +1-254-710-3569 
}
\email{Jameson\_Graber@baylor.edu}

\author{Charafeddine Mouzouni}
\address{C. Mouzouni: Univ Lyon, \'Ecole centrale de Lyon, CNRS UMR 5208, Institut Camille Jordan, 36 avenue Guy de Collonge, F-69134 Ecully Cedex, France.}
\email{mouzouni{@}math.univ-lyon1.fr}

\date{\today}

\begin{abstract}
We investigate a mean field game model for the production of exhaustible resources. In this model, firms produce comparable goods, strategically set their production rate in order to maximise profit, and leave the market as soon as they deplete their capacities.  We examine the related Mean Field Game system and prove well-posedness for initial measure data by deriving suitable a priori estimates. Then, we show that feedback strategies which are computed from the Mean Field Game system provide $\ee$-Nash equilibria to the corresponding $N$-Player Cournot game, for large values of $N$. This is done by showing tightness of the empirical process in the Skorokhod ${\rm M 1}$ topology, which is defined for distribution-valued processes.
\end{abstract}


\maketitle

\section{Introduction} \label{sec:intro}

Since its introduction about ten years ago, the theory of the Mean Field Games has expanded tremendously, and has become an important tool in the study of dynamical and equilibrium behavior of large systems. The theory was introduced separately by a series of seminal papers by Lasry and Lions \cite{lasry06,lasry06a,lasry07} and Caines et al. \cite{Malhame, Malhame2}, and in lectures by Pierre-Louis Lions at the Coll\`ege de France, which were video-taped and made available on the internet \cite{Lions-college-de-France}. The main idea is inspired from statistical physics literature, and consists in considering that a given player interacts with competitors through their statistical distribution in the space of possible states. 

Mean Field Games (MFG) theory provides a methodology to produce approximate Nash equilibria for stochastic differential games with symmetric interactions and a large (but finite) number of players $N$. 
In these games, the exact equilibrium strategies could be determined by a system of coupled Hamilton-Jacobi-Bellman equations, derived from the dynamic programming principle. However, the dimension of the system in general increases in $N$, which makes this system extremely hard to solve either analytically or numerically, especially for large values of $N$. The Mean Field Game approach simplifies the modelling, and allows to compute an approximation of Nash equilibria by solving a system of two forward-backward coupled PDEs.
This simplification justifies partly the interest in the MFG modelling for several applications.

In this paper we revisit a family of MFG models related to competing producers with exhaustible resources.
The dynamic market evolution is driven by the use of certain existing reserves to produce and trade comparable goods. 
Producers disappear from market as soon as they exhaust their capacities, so that the fraction of remaining firms decreases over time.
This type of model was first introduced by Gu\'eant, Lasry, and Lions \cite{gueant2011mean}, and addressed also by Chan and Sircar in \cite{chan2015bertrand}, where it is referred to as ``Bertrand \& Cournot Mean Field Games." In  \cite{chan2015fracking}, the same authors use a similar MFG modelling  approach, to discuss recent changes in global oil market.
A more sophisticated model for the energy industry is proposed recently in \cite{ludkovski2017mean}, where producers have also the possibility to explore new resources to replenish their reserves. 

From a mathematical standpoint, Bertrand \& Cournot MFG system consists in a system of a backward Hamilton-Jacobi-Bellman (HJB) equation to model a representative firm's value function, coupled with a forward Fokker-Planck equation to model the evolution of the distribution of the active firms' states. The exhaustibility condition gives rise to absorbing boundary conditions at zero.
A rigorous analysis of this system was provided in \cite{graber2015existence}, where authors show existence of smooth solutions to the system of equations, and uniqueness under a certain restriction. Unconditional uniqueness is proved in \cite{graber2017variational}, in addition to the analysis of the case with Neumann boundary conditions. 

Otherwise, very little is known so far on the rigorous link between the so called Bertrand \& Cournot MFG models, and the corresponding $N$-Player Bertrand-Cournot stochastic differential games. Indeed, the classical theory cannot be applied to this specific case for two main reasons: on the one hand, because of the absorbing boundary conditions; and on the other hand, because in our model players are coupled through their controls, and therefore belongs to the class of extended Mean Field Games (cf. \cite{bensoussan2013, cardaliaguet2016mfgcontrols, gomes2014extended, gomes2016extended}).
This has motivated the present work, in which we analyse rigorously this question for Cournot competition. 

We investigate the mean-field approximation for $N$-Player continuous-time Cournot game with linear price schedule, and exhaustible resources. In this context, the producers' state variable is the reserves level, and the strategic variable is the rate of production. Producers disappear from the market as soon as they deplete their reserves, and the remaining active producers are constrained to set a non-negative rate of production, in order to manage their remaining reserves and maximize sales profit. Due to this constraint on the production rate -- which is natural from a modeling view point -- we obtain a Hamiltonian function that is less regular in comparison to \cite{graber2015existence, chan2015bertrand, graber2017variational}. Market demand is assumed to be linear, so that the received market price is a non-increasing linear function of the total production across all producers. Further details and explanations about the model will be given in Section \ref{sec:main results}.

We start by studying the 
resulting system of  coupled PDEs (the MFG system) by deriving suitable a priori estimates in H\"older spaces. We shall assume that the initial data is a probability measure that is supported on $(0, L]$, for some $L>0$, which entails that all producers start with positive reserves. Hence, our analysis completes that which is found in \cite{graber2015existence, graber2017variational}, by treating the case of a less regular Hamiltonian function and initial measure data. 
Next, we prove that the feedback control given by the solution of the Mean Field Game system provides an $\ee$-Nash equilibrium (cf. Definition \ref{epsilon-Nash}) to the corresponding $N$-Player Cournot game, where the error $\ee$ is arbitrary small for large enough $N$. 
This result shows that the MFG model is indeed a good approximation to the game with finitely many players, and reinforces numerical methods based on the MFG approach. 
As in the classical theory, the key argument in the proof of this result is a suitable law of large numbers. In our context, the main mathematical challenge comes from the fact that agents interact through the boundary behaviour, and are coupled by means of their chosen production strategies.
To prove a tailor-made law of large numbers, we employ a compactness method borrowed from \cite{ledger2016-2, ledger2016-1}, by showing tightness of the empirical process in the space of distribution valued \emph{c\`adl\`ag} processes, endowed with Skorokhod's M1 topology  \cite{ledger2016-2}. 
In contrast to the classical tools used so far, this method does not provide an exact quantification of the error $\ee$, which is its main downside. Nevertheless, this approach has proven to be convenient for studying systems with absorbing boundary conditions.
We also believe that it could be extended to the case of a systemic common noise, just as \cite{ledger2016-2} contains an analysis of a \emph{stochastic} McKean-Vlasov equation.
However, we do not address this case here, finding the analysis of the stochastic HJB/FP-system somewhat out of reach under our assumptions on the data
(cf.~\cite[Section 4]{master-equation} and the hypotheses found there).

For background on Skorokhod's topologies for real valued processes, we refer the reader to \cite{Whitt2002} and references therein. The M1 topology is extended to the space of tempered distributions, and to more general spaces in \cite{ledger2016-2}.
The fact that the feedback MFG control provides $\ee$-Nash equilibria for the corresponding differential games with a large (but finite) number of players, was first noticed by Caines et al.~\cite{Malhame, Malhame2} and further developed in several works (see e.g. \cite{carmona2013, kolokoltsov2011} among many others). Cournot games with exhaustible resources and finite number of agents is investigated by Harris et al.~in \cite{harris2010}, and the corresponding MFG models were studied in \cite{gueant2011mean, chan2015bertrand,chan2015fracking, ludkovski2017mean} with different variants, and numerical simulations. 
We refer the reader to \cite{gueant2011mean, bensoussan2013, carmona2018, lasry07} for further background on Mean Field Game theory.

The paper is organized as follows.
In the remainder of this section, we give some technical notations and preliminaries, introduce the mathematical description of the $N$-Player Cournot game with limited resources and the corresponding Mean Field Game, and state the main results of this paper.
In Section \ref{section2} we prove existence and uniqueness of regular solutions to the MFG system by deriving suitable H\"older estimates.
In Section \ref{section::application} we show that the feedback control computed from the MFG system is an $\ee$-Nash equilibrium to the $N$-Player game. For that purpose, we start by showing the weak convergence of the empirical process with respect to the M1 topology, then we deduce the main result by recalling the interpretation of the MFG system in terms of games with mean-field interactions.

\subsection{Notations and preliminaries} \label{sec:notation}
Throughout this article we fix $L>0$, define $Q:=(0,L)$, and $Q_T := (0,T) \times (0,L)$. 
For any domain $\D$ in $\bb{R}$ or $\bb{R}^2$ we define $\bar \D$ to be the closure of $\D$, 
$L^s(\D)$, $1\leq s \leq \infty$ to be the Lebesgue space of $s$-integrable functions on $\D$;  $L^s(\D)_+$ to be the set of elements  $w\in L^s(\D)$ such that $w(x)\geq 0$ for a.e. $x \in \D$;
$W_s^{k}(\D)$,  $k\in \N$, $1\leq s \leq \infty$, to be the Sobolev space of functions  having a weak derivatives up to order $k$ which are $s$-summable on $\D$; 
$\C(\D)$ to be the space of all continuous functions on $\D$;
$\C_0(\D)$ to be the space of all continuous functions on $\D$ that vanish at infinity ($\C_0(\D) = \C(\D)$ when $\D$ is compact);
 $\C^{\t}(\D)$ to be the space of all H\"older continuous functions with exponent $\t$ on $\D$;  
$\C_c^{\infty}(\D)$ to be the set of smooth functions whose support is a compact included in $\D$;  $\S_\R$ denotes the space of rapidly decreasing functions, and $\S'_\R$ the space of tempered distributions.

For a subset $\D \subset \overline{Q_T}$, we also define $\C^{1,2}(\D)$ to be the set of all functions on $\D$ which are locally continuously differentiable in $t$ and twice locally continuously differentiable in $x$,
and by $W_s^{1,2}(\D)$ the space of elements of $L^s(\D)$ having weak derivatives of the form $\p_t^{j}\p_x^{k}$ with $2j+k \leq 2$, endowed with the following norm:
$$
\| w \|_{W_s^{1,2}} :=  \sum_{2j+k\leq 2} \| \p_t^{j}\p_x^{k}w \|_{L^s}.
$$
The space of $\R$-valued Radon measures on $\D$ is denoted $\M(\D)$, which we identify with $\C_0(\D)^*$ endowed with weak$^*$ topology, and $\P(\D)$, $\tilde\P(\D)$ are respectively the convex subset of probability measures on $\D$, and the convex subset of sub-probability measures: that is the set of positive radon measures $\mu$, s.t. $\mu(\D)\leq 1$. For any measure $\mu \in \M(\D)$, we denote by ${\rm supp}(\mu)$ the support of $\mu$. 

Throughout the paper, we fix a complete filtered probability space $(\Omega, \F, \mathbb F=(\F_t)_{t\geq 0} , \Pp)$, and suppose that is rich enough to fulfill the assumptions that will be formulated in this article.  We also fix constants $r, \sigma,T>0$,
and denote by $C$ a {\em generic} constant whose precise value may change from line to line.  We also use the notation $C(\a,\b,\g)$ and the like to point out the dependence of some constant on parameters $\a,\b, \g$. Moreover, we use the notation $X\sim  \mu$ to define a random variable $X$ with law $\mu$.
For any $\R$-valued function $w$ we define the positive and negative parts of $w$, respectively:
$$
w^+ := \frac{1}{2}(|w|+w), \quad \mbox{ and } \quad w^- := \frac{1}{2}(|w|-w);
$$
and for any $x,y \in \R$ we use the following notation for the minimum and maximum, respectively:
$$
x \wedge y := \frac{1}{2}\left(x+y -|x-y| \right); \quad \mbox{ and } \quad x \vee y := \frac{1}{2}\left(x+y+|x-y| \right).
$$

Let us recall a few basic facts on stochastic differential equation with reflecting boundary in a half-line. Given a random variable $V$ that is supported on $(-\infty, L]$, we look for a pair of a.s. continuous and adapted processes $(X_t)_{t\geq0}$ and  $(\xi_t^X)_{t\geq 0}$ such that:
\begin{subequations}
\begin{eqnarray}\label{skorokhod-problem}
&&X_t =  V + \int_0^t b(s,X_s) \dd s +\sigma W_t - \int_0^t \1_{\{X_s=L\}} \dd \xi_s^X \in (-\infty, L],  \nonumber \\
&&\xi_t^X = \int_0^t \1_{\{X_s=L\}} \dd \xi_s^X,\\ \nonumber
&&X_0=V, \ \  \xi_0^X=0, \  \ \mbox{ and } \xi^X \mbox{ is nondecreasing,}
\end{eqnarray}
where $(W_t)_{t\geq 0}$ is a $\mathbb F$-Wiener process  that is independent of $V$.
The random process $(X_t)_{t\geq 0}$ is the \emph{reflected diffusion}, $(\xi_t^X)_{t\geq 0}$ is the \emph{local time}, and the above set of equations is called the \emph{Skorokhod problem}. Throughout the paper, we shall write problem \eqref{skorokhod-problem} in the following simple form:
$$
\dd X_t =   b(t,X_t) \dd t + \sigma \dd W_t - \dd \xi_t^X,  \quad X_0=V.
$$
Suppose that the function  $b$ is bounded, and satisfies for some $K>0$ the following condition:
\begin{equation}\label{assumption-SDEs}
\left| b(t,x)-b(t,y) \right| \leq K |x-y|
\end{equation}
for all  $t\in [0,T]$, and $ x,y \in (-\infty, L]$. Then, it is well-known (see e.g. \cite{anderson1976, Reflexion}) that under these conditions, problem \eqref{skorokhod-problem} has a unique solution on $[0,T]$. Moreover, this solution is given explicitly by:
\begin{equation} \label{skorokhod-problem-equation-formula}
X_t:= \Gamma_t(Y), \quad \xi_t^X:=Y_t- \Gamma_t(Y);
\end{equation}
where the process $(Y_t)_{t\in [0,T]}$ is the solution to
\begin{equation}\label{skorokhod-problem-equation}
Y_t =  V + \int_0^t b(s,\Gamma_s(Y) ) \dd s +\sigma W_t,
\end{equation}
and where $\Gamma$ is the so called \emph{Skorokhod map}, that is given by
$$
\Gamma_t(Y) := Y_t  -\sup_{0\leq s\leq t}\left( L-Y_s \right)^{-}.
$$
Furthermore, notice that 
\begin{equation}\label{inequalities-growth-T}
\xi_t^X-\xi_{t+h}^X  \geq \inf_{v \in [0,h]}(Y_t - Y_{t+v})
\end{equation}
for any $t\in[0,T)$ and $h\in(0,T-t)$.
In fact, when $\xi_t^X < \xi_{t+h}^X$, then
$$
0< \xi_{t+h}^X :=\sup_{0\leq s\leq t+h}\left( L-Y_s \right)^{-} = \sup_{t\leq s\leq t+h}\left( L-Y_s \right)^{-} = (Y_{v_0}-L)
$$
for some $t\leq v_0\leq t+h$. Therefore
\begin{eqnarray*}
\xi_t^X-\xi_{t+h}^X &=& \sup_{0\leq s\leq t}\left( L-Y_s \right)^{-} - \sup_{0\leq s\leq t+h}\left( L-Y_s \right)^{-} \\
&\geq& (Y_t-L)-(Y_{v_0}-L) \geq \inf_{v \in [0,h]}(Y_t - Y_{t+v}).
\end{eqnarray*}
This entails \eqref{inequalities-growth-T} since the last inequality still holds when $\xi_t^X = \xi_{t+h}^X$.
\end{subequations}

Now we consider a boundary value problem for the \emph{Fokker-Planck} equation.
Let $b$ in $L^2(Q_T)$, $m_0 \in \P(\bar Q)$, and consider the following Fokker-Planck equation
\begin{subequations}
\begin{equation}
\label{eq:fp}
\left\{
\begin{aligned}
&m_t - \frac{\sigma^2}{2} m_{xx} - (b m)_x = 0 \ \ \mbox{ in } Q_T \\
& m(0)=m_0 \ \ \mbox{ in } Q,
\end{aligned}
\right.
\end{equation}
complemented with the following mixed boundary conditions:
\begin{equation}\label{eq:fp:bc}
m(t,0)=0, \ \  \mbox{ and } \ \ \frac{\sigma^2}{2}m_x(t,L) + b(t,L)m(t,L)=0 \quad \mbox{ on } (0,T).
\end{equation}
Then we define a \emph{weak solution} to \eqref{eq:fp}-\eqref{eq:fp:bc} to be a function $m \in L^1(Q_T)_+$ such that $m|b|^2$ in $L^1(Q_T)$, and
\begin{equation}
\label{eq:weak fp}
\int_0^T \int_0^L m(-\phi_t - \frac{\sigma^2}{2} \phi_{xx} + b \phi_x) \dd x \dd t = \int_0^L \phi(0,.)\dd m_0
\end{equation}
for every $\phi \in \C_c^\infty([0,T) \times \overline{Q})$ satisfying
\begin{equation}\label{eq:weak fp:bc}
\phi(t,0)=\phi_x(t,L) = 0, \ \ \forall t \in (0,T).
\end{equation}
\end{subequations}
This is the definition given by Porretta in \cite{porretta2015weak}.
The only difference is that here we consider mixed boundary conditions and measure initial data. 

When  $m_0\in L^1(Q)_+$, the problem \eqref{eq:fp} endowed with periodic, Dirichlet or Neumann boundary conditions has several interesting features that were pointed out in \cite[Section 3]{porretta2015weak}.
In particular, they are unique \cite[Corollary 3.5]{porretta2015weak} and enjoy some extra regularity \cite[Proposition 3.10]{porretta2015weak}.
 Note that these results still hold in the case of mixed boundary conditions \eqref{eq:fp:bc}. Throughout the paper, we shall use the results of \cite[Section 3]{porretta2015weak} for \eqref{eq:fp}-\eqref{eq:fp:bc}.

In the case where $b$ is bounded, we shall use the fact that \eqref{eq:fp}-\eqref{eq:fp:bc} admits a unique weak solution, for any $m_0 \in \P(\bar Q)$. In fact, one can construct a solution by considering a suitable approximation of $m_0$, and then use the compactness results of \cite[Proposition 3.10]{porretta2015weak} in order to pass to the limit in $L^1(Q_T)$. The uniqueness is obtained by considering the dual equation, and using the same steps as for \cite[Corollary 3.5]{porretta2015weak} (cf. Proposition \ref{uniqueness:measure-valued}).

\subsection{Mathematical description of the model and main results}\label{sec:main results}
Let us now give a precise description of the problems considered in this paper.
Consider a market with $N$ producers of a given good whose strategic variable is the rate of production, and where raw materials are in limited supply. Concretely, one can think of energy producers that use exhaustible resources, such as oil, to produce and sell energy. Firms disappear from the market as soon as they deplete their reserves of raw materials.

Let us formalize this model in precise mathematical terms. Let $\left( W^{j} \right)_{1\leq j \leq N}$ be a family of $N$ independent $\mathbb F$-Wiener processes on $\R$, and consider the following system of Skorokhod problems:
\begin{equation}
\label{initial-dynamics-cournot}
\left\{
\begin{aligned}
&\dd X_{t}^{i} =  - q_t^i \dd t + \sigma  \dd W_{t}^{i}-\dd \xi_t^{X^{i}},\\
&X_{0}^{i} = V_{i},  \quad i=1,...,N.
\\\end{aligned}
\right.
\end{equation} 
Here $(V_{1},...,V_{N})$ is a vector of i.i.d and $\F_0$-measurable random variables, and we assume that  $V_1,...,V_N$ are independent of $W^1,...,W^N$ respectively.
Let us fix a common horizon $T>0$, and set 
$$
\tau^{i}:= \inf\left\{ t\geq 0 \ : \ X_t^{i} \leq 0 \right\} \wedge T.
$$
The stopped random process $\left(X_{t \wedge \tau^{i}}^{i} \right)_{t\in[0,T]}$ models the reserves level of the $i^{th}$ producer on the horizon $T$, which is gradually depleted according to a non-negative controlled rate of production $\left(q_t^i \right)_{t\in[0,T]}$. The stopping condition indicates that a firm
can no longer replenish its reserves once they are exhausted, and
the Wiener processes in \eqref{initial-dynamics-cournot} model the idiosyncratic fluctuations related to production. In addition, we consider $L$ to be an upper bound on the reserves level of any player. This latter assumption is also considered in \cite{graber2015existence, graber2017variational} and is taken into account by considering reflected dynamics in \eqref{initial-dynamics-cournot}. Since the rate of production is always non-negative, note that reflection has practically no effect when $L$ is large compared to the initial reserves and $\sigma$.
\begin{remark}
	[State constraints]
	Instead of reflecting boundary conditions, one could insist upon a hard state constraint of the form $X_t^i \leq L$.
	Some recent work on MFG with state constraints suggests this is possible \cite{cannarsa2018existence,cannarsa2018constrained,cannarsa2018mean}, provided one correctly interprets the resulting system of PDE (the Fokker-Planck equation presents a special challenge).
	In this work we take a more classical approach, for which probabilistic tools are more readily available.
\end{remark}

The producers interact through the market. We assume that demand is linear, so that the price $p^i$ received by the firm $i$ reads:
\begin{equation}\label{price-expression-cournot}
p_t^i = 1-(q_t^i + \k \bar q_t^i), \quad \mbox{ where } \quad \bar q_t^i = 
\frac{1}{N-1}\sum_{ j\neq i} q_t^{j}\1_{t<\tau^{j}}, \ \ \mbox{ for } \ \ 0\leq t\leq T.
\end{equation}
Here $\k>0$ expresses the degree of market interaction, in proportion to which abundant total production will put downward pressure on all the prices. Note that only firms with nonempty reserves at $t \in[0,T]$ are taken into account in \eqref{price-expression-cournot}. The other firms are no longer present on the market.
The producer $i$ chooses a non-negative production rate $q^i$ in order to maximize the following discounted profit functional:
$$
\J_c^{i,N}(q^1,...,q^N):=\E\left\{ \int_{0}^{T} e^{-rs}  \left( 1-\k \bar q_s^i -q_s^i \right) q_s^i \mathds{1}_{s < \tau^{i}} \dd s + e^{-rT}u_{T}( X_{\tau^{i}}^{ i} )  \right\},
$$
where the terminal profit $u_T$ is a smooth function satisfying $u_T(0)=0$.
Observe that firms can no longer earn revenue as soon as they deplete their reserves. We refer to \cite{harris2010, chan2015fracking} for further explanations on the economic model and applications.

We denote by $\mathbb A_c$ the set of admissible controls for any player; that is the set of \emph{Markovian feedback controls}, i.e.  
$
q_t^{i}=q^i \left(t, X_t^{1},...,X_t^{N} \right);
$ 
such that $(q_t^i)_{t\in[0,T]}$ is non-negative, satisfies
$$
\E\left[\int_0^T |q_s^{i} |^2 \mathds{1}_{s < \tau^{i}} \dd s \right] < \infty,
$$
and the $i^{th}$ equation of \eqref{initial-dynamics-cournot} is well-posed in the classical sense. Restriction to Markovian controls  rules out equilibria with undesirable properties such as non-credible threats (cf. \cite[Chapter 13]{Tirole1991}).

Now, we give a definition of Nash equilibria to this game:

\begin{definition}[Nash equilibrium]   \label{Nash:equilibrium-Cournot}
	A strategy profile $\left(q^{1,\ast},..., q^{N, \ast}\right)$ in $\prod_{i=1}^N \mathbb A_c$ is a \emph{Nash equilibrium} of the $N$-Player Cournot game, if for any $i=1,...,N$ and $q^i \in \mathbb A_c$ 
	$$
	\J_c^{i,N}\left(q^i; ( q^{j,\ast})_{j\neq i} \right) \leq \J_c^{i,N}\left( q^{1,\ast},..., q^{N,\ast}\right).
	$$
\end{definition}
In words, a Nash equilibrium is a set of admissible strategies such that each player has taken an optimal trajectory in view of the competitors' choices. 

The existence of Nash equilibria for the $N$-Player Cournot game with exhaustible resources is addressed in \cite{harris2010}. In particular, the authors show the existence of a unique Nash equilibrium in the static (one period) case, and study numerically a specific duopoly example by using a convenient asymptotic expansion. In general, the analysis of equilibria  for $N$-Player Cournot games is a challenging task both analytically and numerically, especially when $N$ is large. In the case of exhaustible resources, the dynamic programming principle generates an even more complex PDE system because of the nonstandard boundary conditions which are obtained (cf. \cite[Section 3.1]{harris2010}). 

To remedy this problem several works have rather considered a Mean-Field model \cite{gueant2011mean, ludkovski2017mean, harris2010, chan2015bertrand, chan2015fracking} as an approximation to the initial $N$-Player game, when $N$ is large.
	More precisely, we introduce the following:
	\begin{definition}[$\ee$-Nash equilibrium] \label{epsilon-Nash}
		Let $\ee>0$, and let $(\hat q^1,...,\hat q^N)$ be an admissible strategy profile (i.e.~an element of $\prod_{i=1}^N \mathbb A_c$).
		We say $(\hat q^1,...,\hat q^N)$ provides an \emph{$\ee$-Nash equilibrium} to the game $\J_c^{1,N},...,\J_c^{N,N}$ provided that, for any $i=1,...,N$ and $q^i \in \mathbb A_c$,
		\begin{equation*}
		\J_c^{i,N}\left(q^i; (\hat q^j)_{j\neq i} \right) \leq \ee+\J_c^{i,N}\left(\hat q^1,...,\hat q^N \right).
		\end{equation*}
	\end{definition}
	In words, an $\ee$-Nash equilibrium is a set of admissible strategies such that each player has taken an \emph{almost} optimal trajectory in view of the competitors' choices, where $\ee$ measures the distance from optimality.

	To construct $\ee$-Nash equilibria, we turn to the corresponding \emph{mean field} problem.
	Let us now consider a continuum of agents, producing and selling comparable goods. At time $t=0$, all the players have a positive capacity $x\in (0,L]$, and are distributed on $(0,L]$ according to $m_0$. 

The remaining capacity (or reserves) of any atomic producer with a production rate $(\rho)_{t\geq 0}$ depletes according to 
$$
\dd X_t^\rho=- \rho_t \1_{t<\tau^\rho} \dd t + \sigma \1_{t<\tau^\rho} \dd W_{t} - \dd \xi_t^{X^\rho},
$$
where 
$$
\tau^\rho := \inf\{t\geq 0 : X_t^\rho \leq 0\} \wedge T,
$$
and $(W_t)_{t\in[0,T]}$ is a $\mathbb F$-Wiener process.
A generic player which anticipates the total production $\bar q$ expects to receive the price
$$
p:= 1-(\k \bar q + \rho)
$$
and solves the following optimization problem:
\begin{equation}\label{the functional-defintion}
\max_{\rho\geq 0 } \J_c(\rho):= \max_{\rho\geq 0 } \E\left\{\int_0^T e^{-r s} \left(1-\k \bar q_s-\rho_s  \right)\rho_s \1_{s<\tau^{\rho}}\dd s + e^{-rT}u_{T}\left(X_{T}^\rho \right)   \right\}.
\end{equation}
The maximum in \eqref{the functional-defintion} is taken over all $\mathbb F$-adapted and non-negative processes $(\rho_t)_{t\in [0,T]}$ such that  
$$
\E\left[\int_0^T |\rho_s|^2 \1_{s<\tau^\rho} \dd s \right] <\infty 
$$ 
and
$(X_t^\rho)_{t\in[0,T]}$ exists in the classical sense.

According to MFG theory, the equilibrium in this setting can be computed by solving the following coupled system of parabolic partial differential equations:
	\begin{equation}
	\label{MFG-limit-1} 
	\left\{
	\begin{aligned}
	&u_{t}+\frac{\sigma^{2}}{2} u_{xx}-r u + q_{u,m}^2=0
	\quad \mbox{ in }Q_T,\\
	& m_{t}-\frac{\sigma^{2}}{2}m_{xx}-\left\{ q_{u,m} m \right\}_{x}=0 \quad \mbox{ in }Q_T,\\
	& m(t,0)=0,\quad  u(t,0)=0,\quad  u_{x}(t,L)=0\quad \mbox{ in } (0,T),  \\
	&m(0)=m_{0}, \quad  u(T,x) = u_{T}(x),\quad \mbox{ in } [0,L],\\
	&\frac{\sigma^{2}}{2}m_{x} + q_{u,m}m=0    \quad \mbox{ in } (0,T)\times\{L\},
	\end{aligned}
	\right.
	\end{equation} 
	where the function $q_{u,m}$ involved in the system is given by:
	\begin{equation}\label{MFG-quantity-def-initial}
	q_{u,m}(t,x):= \frac{1}{2} \left( 1-\k \bar q (t) -u_{x}(t,x)  \right)^{+},  
	\quad 
	\mbox{where}
	\quad
	\bar q(t) := \int_0^L q_{u,m}(t,x)m(t,x)\dd x,
	\end{equation}
	and $\k>0$.
	Here $m$ is the density of a continuum of market actors,
	$q_{u,m}(t,x)$ is the optimal production rate of an atomic player with reserves $x$ at time $t$,  and $u$ is the the game value function of an atomic player following the production policy $q_{u,m}$. 
	
	Let us assume that $u_T$ is a function in ${\C}^2(\bar Q)$,  such that the first derivative of $u_T$ denoted by $u'_T$ fulfils
	\begin{equation}  \label{A1}
	\tag{$\mathcal H$1}
	u'_{T}\geq 0 \quad  \mbox{and} \quad u_T(0)=u'_{T}(L)=0
	\end{equation}
	and that $m_0$ is a probability measure with support away from 0, i.e.
	\begin{equation} \label{Aa1}
	\tag{$\mathcal H$2}
	m_0 \in \P(\bar Q),  \ \mbox{ and } \  {\rm supp} (m_0) \subset (0,L].
	\end{equation}
	We shall say that a pair $(u,m)$ is a solution to \eqref{MFG-limit-1}, if 
	\begin{enumerate}[ label=(\roman*)]	
		\item $ u \in \C^{1,2}(Q_T)$,  $u,u_x \in \C(\overline {Q_T})$; 
		\item $m \in \C([0,T]; \M(\bar Q)) \cap L^1(Q_T)_+$, and $\| m(t) \|_{L^1} \leq 1$ for every  $t\in (0,T]$; 
		\item the equation for $u$ holds in the classical sense, while the equation for $m$ holds in the weak sense \eqref{eq:weak fp}.
	\end{enumerate}
	
	The following lemma establishes the connection between $(u,m)$ and problem \eqref{the functional-defintion}:
	\begin{lemma}\label{Lemma-optimal:control-Cournot}
		Let $(u,m)$ be a solution to \eqref{MFG-limit-1} and set $\rho_t^\ast:= q_{u,m}(t,X_t^{\rho^\ast})$.
		Then
		\begin{equation}\label{equation-1-lemma211-cournot}
		\max_{\rho \geq 0} \J_c(\rho) = \J_c(\rho^\ast)= \int_0^L u(0,.) \dd m_0.
		\end{equation}
	\end{lemma}
	The proof of Lemma \ref{Lemma-optimal:control-Cournot} is standard, and is given in Appendix \ref{App1}. 
	We deduce that the MFG system \eqref{MFG-limit-1} describes an equilibrium configuration for a Cournot game with exhaustible resources and a continuum of producers.
	
	We are now in a position to state the main results of this paper.
	\begin{theorem}[Well-posedness]\label{theorem:wellposedness}
		There exists a unique solution $(u,m)$ to system \eqref{MFG-limit-1}. 
	\end{theorem}
	\begin{theorem}[Existence of $\ee$-Nash equilibria]\label{Theorem-epsilon::Nash-Cournot} Let $(u,m)$ be the solution to the MFG system \eqref{MFG-limit-1}, and let $q_{u,m}$ be given by \eqref{MFG-quantity-def-initial}. 
		For any $N\geq 1$ and $i\in\{1,...,N\}$ let
		\begin{equation}
		\left\{
		\begin{aligned}
		&\dd \hat X_{t}^{i} =  - q_{u,m}(t, \hat X_{t}^{i})\dd t + \sigma  \dd W_{t}^{i}-\dd \xi_t^{\hat X^{i}}\\
		&X_{0}^{i} = V_{i}, 
		\\\end{aligned}
		\right.
		\end{equation}
		and set $\hat q_t^i:= q_{u,m}(t, \hat X_{t}^{i})$.  
		Then for any $\ee>0$, the strategy profile $(\hat q^1,...,\hat q^N)$ is admissible, i.e. belongs to $\prod_{i=1}^N \mathbb A_c$, 
		and provides an $\ee$-Nash equilibrium to the game $\J_c^{1,N},...,\J_c^{N,N}$ for large $N$. Namely: 
		$\forall \ee>0, \ \exists N_\ee \geq 1$ such that
		\begin{equation}\label{epsilon-Nash-th2.13-Cournot}
		\forall N\geq N_\ee, \forall i=1,...,N, \quad    \J_c^{i,N}\left(q^i; (\hat q^j)_{j\neq i} \right) \leq \ee+\J_c^{i,N}\left(\hat q^1,...,\hat q^N \right),
		\end{equation}
		for any admissible strategy $q_i \in \mathbb A_c$. 
	\end{theorem}
	The proof of Theorem \ref{theorem:wellposedness} is given in Section \ref{section2}, while Section \ref{section::application} is devoted to the proof of Theorem \ref{Theorem-epsilon::Nash-Cournot}.

\subsection*{Acknowledgement}
The first author was supported by the National Science Foundation under NSF Grant DMS-1612880.
The second author was supported by LABEX MILYON (ANR-10-LABX-0070) of Universit{\'e} de Lyon, within the program "Investissements d'Avenir" (ANR-11-IDEX-0007) operated by the French National Research Agency (ANR), and partially supported by project (ANR-16-CE40-0015-01) on Mean Field Games.


\section{Analysis of the MFG system}\label{section2}
This section is devoted to the analysis of \eqref{MFG-limit-1}.
We begin by establishing some preliminary estimates having to do with (i) the existence and regularity of $q_{u,m}$ defined in \eqref{MFG-quantity-def-initial}, and (ii) some regularity properties of weak solutions to the Fokker-Planck equation.
Then we prove some a priori bounds on solutions of the system \eqref{MFG-limit-1}, which gives way to our proof of Theorem \ref{theorem:wellposedness} by means of classical fixed point theory.

\subsection{Preliminary estimates} \label{section2: Preliminary estimates}
We start by giving an alternative convenient expression for the production rate function $q_{u,m}$. We aim to write $q_{u,m}$ as a functional of $u_x, m$ and the market price function $p_{u,m}$, that is defined by  \cite{chan2015bertrand}:
\begin{equation}\label{MFG-price-def-initial}
p_{u,m}(t,x) := 1-(q_{u,m}(t,x)+\k \bar q (t)).
\end{equation}
The latter expression means that the price $p_{u,m}(t,x)$ received by an atomic player with reserves $x$ at time $t$, is a linear and nonincreasing function, of the player's production rate $q_{u,m}(t,x)$, and the aggregate production rate across all producers $\bar q(t)$. For any $\mu \in \M(\bar Q)$, we define
\begin{equation}\label{coefficients-defintion:a,c,eta}
a(\mu) := \frac{1}{1+\k\eta(\mu)}; \quad c(\mu):=1-a(\mu); \quad \eta(\mu):= \int_0^L \dd |\mu|
\end{equation}
and set
\begin{subequations}
\begin{equation}\label{expression::market::price-0}
\overline{p}(t) := \frac{1}{\eta(m(t))}\int_0^L p_{u,m}(t,x)m(t,x) \dd x.
\end{equation}
By integrating \eqref{MFG-price-def-initial} with respect to $m$ and after a little algebra one recovers the following identity
\begin{equation*}
a\left( m(t) \right)+c\left( m(t) \right)\overline{p}(t) = 1 - \k \bar{q}(t),
\end{equation*}
which entails
\begin{equation}\label{expression::market::price}
p_{u,m}(t,x) = a\left( m(t)\right)+c\left( m(t) \right)\overline{p}(t) -q_{u,m}(t,x),
\end{equation}
and
\begin{equation}\label{MFG-quantity-def}
q_{u,m}(t,x)= \frac{1}{2} \left\{ a\left(m(t)\right)+c\left( m(t) \right)\overline{p}(t) -u_{x}(t,x)  \right\}^{+}.
\end{equation}
\end{subequations}
This duality is also known as \emph{Bertrand and Cournot equivalence}, and expresses the fact that the problem of controlling the rate of production by anticipating global production, is equivalent to the problem of controlling the selling price by anticipating the average price in the market and the rate of active producers. We omit the details and refer to \cite[Section B.2]{chan2015bertrand}.  For convenience, we shall often use \eqref{MFG-quantity-def} as a definition for $q_{u,m}$.

In contrast to the systems studied in \cite{chan2015bertrand, graber2015existence, graber2017variational}, $p_{u,m}$ has no explicit formula and is only defined as a fixed point through \eqref{expression::market::price-0}-\eqref{MFG-quantity-def}. The following Lemma makes that statement clear and point out a few facts on the market price function.

\begin{lemma}\label{fixed::point::market-price}
Let $u\in L^{\infty}\left(0,T; \C^1(\bar Q)\right)$, $m \in L^1(Q_T)_+$, and $\k>0$. Then the market price function $p_{u,m}$ is well-defined through \eqref{expression::market::price-0}-\eqref{MFG-quantity-def},  belongs to $L^{\infty}(0,T; \C(\bar Q))$, and satisfies 
\begin{equation}\label{estimation:price-1}
-\| u_x\|_{\infty} \leq p_{u,m}\leq 1.
\end{equation}
Moreover, if $u_{x}$ is non-negative, then $p_{u,m}$ is non-negative as well. 
\end{lemma}
\begin{proof}
Let $f:\bb{R}^2 \to \bb{R}$ be given by $f(x,y) = x - \frac{1}{2}(x-y)^+$. Note that $f$ is 1-Lipschitz in the first variable, and $\frac{1}{2}$-Lipschitz in the second. 
For any $p,w \in \X:=L^{\infty}(0,T; \C(\bar Q))$, define
$$
\ell(m,p)(t):= a\left( m(t) \right)+c\left( m(t) \right)\overline{p}(t),\quad \mbox{ where }\quad
\overline{p}(t):=\frac{1}{\eta(m(t))}\int_0^{L} p(t,x) m(t,x) \dd x,
$$
and 
$$\La(w,m,p)(t,x) := f(\ell(m,p)(t),w(t,x)).$$
We note the following inequalities for future reference:
\begin{subequations}
\begin{align}
\label{eq:ell p}
\left|\ell(m,p)(t)-\ell(m,p')(t)\right| &\leq \frac{\k}{1+\k}\left\|p(t,\cdot)-p'(t,\cdot)\right\|_\infty,\\
\label{eq:F p}
\left\|\La(w,m,p)(t,\cdot)-\La(w,m,p')(t,\cdot)\right\|_\infty &\leq \frac{\k}{1+\k}\left\|p(t,\cdot)-p'(t,\cdot)\right\|_\infty,\\
\label{eq:F w}
\left|\La(w,m,p)(t,x)-\La(w',m,p)(t,x)\right| &\leq \frac{1}{2}\left|w(t,x)-w'(t,x)\right|,\\
\label{eq:F m}
\left|\La(w,m,p)(t,x)-\La(w,m',p)(t,x)\right| &\leq \left|\ell(m,p)(t) - \ell(m',p)(t)\right|.
\end{align}
\end{subequations}

We aim to use Banach fixed point Theorem to show that
\begin{equation}\label{fixed-point-problem:lemma21}
p=a(m)+c(m)\bar p -\frac{1}{2}\left\{a(m)+c(m)\bar p -u_x \right\}^+
\end{equation}
has a unique solution $p_{u,m} \in \X$, which satisfies \eqref{estimation:price-1}.
For any $p \in \X$, let us set
$$
\psi(p):= \La(u_x,m,p)=a(m)+c(m)\bar p -\frac{1}{2}\left\{a(m)+c(m)\bar p -u_x \right\}^+.
$$
Observe that $\psi(\X)\subset \X$, and  $p\leq 1$ entails $\psi(p)\leq 1$. Moreover, if we suppose that $p\geq -\| u_x\|_{\infty}$, then it holds that
$$
\psi(p) \geq -c(m)\| u_x\|_{\infty},
$$
so that $\psi(p)\geq -\| u_x\|_{\infty}$, since $c(m)<1$.
On the other hand, by appealing to \eqref{eq:F p} we have
$$
\left\| \psi(p_{1})-\psi(p_{2}) \right\|_{\X} \leq  \frac{\k}{1+\k} \left\| p_1-p_2  \right\|_{\X} \quad \forall p_1,p_2 \in \X.
$$
Therefore by invoking Banach fixed point Theorem, and the estimates above we deduce the existence of a unique solution $p_{u,m} \in \X$ to problem \eqref{fixed-point-problem:lemma21} satisfying \eqref{estimation:price-1}.

When $u_x$ is non-negative, note that $ p\geq 0$ entails $\psi(p)\geq 0$, so that the same fixed point argument yields
$ p_{u,m}\geq 0$.
\end{proof}

Next, we collect some facts related to the Fokker-Planck equation \eqref{eq:fp}-\eqref{eq:fp:bc}.

\begin{lemma}[regularity of $\eta$] \label{continuity of measure}
	Let $m$ be a weak solution to \eqref{eq:fp}-\eqref{eq:fp:bc}, starting from some $m_0$ satisfying \eqref{Aa1}. 
	Suppose that $b$ is bounded, 
	and satisfies \eqref{assumption-SDEs}. Then the map $t\to \eta(t):=\eta(m(t))$ is continuous on $[0,T]$.
	
	Moreover, if in addition $m_0$ belongs to $L^1(Q)$, then we have:
	\begin{subequations}
	\begin{enumerate}[ label=(\roman*)]	
	\item the function $t \to \eta(t)$ is locally H\"older continuous on $(0,T]$; namely, 
	 there exists $\gamma>0$ such that
	 \begin{equation} \label{eq:eta holder}
	\left|\eta(t_1)-\eta(t_2) \right| \leq C(t_0,\|b\|_\infty) \left|t_1-t_2 \right|^{\gamma} \quad \forall t_1,t_2 \in [t_0,T]
	\end{equation}
	for all $t_0 \in (0,T)$;
	\item for any $\a>0$ and $\phi \in \C^{\a}(\bar Q)$, there exists $\b>0$ such that
	\begin{equation} \label{eq: holder-dual holder}
	\left| \int_0^L \phi(x) \left(m(t_1,x)-m(t_2,x) \right) \dd x \right| 
	\leq C(t_0,\|b\|_\infty, \| \phi\|_{\C^{\a}}) |t_1-t_2|^{\b} \quad \forall t_1,t_2 \in [t_0,T]
	\end{equation}
	for all $t_0 \in (0,T)$.
	\end{enumerate}
	\end{subequations}
\end{lemma}

\begin{remark}
	This lemma shows that $t \mapsto m(t)$ is locally H\"older continuous in time in $(0,T]$ with respect to the $(\C^\alpha)^*$ topology; this is useful later to get equicontinuity for construction of a fixed point (cf.~Section \ref{section 2: Well-posedness}).
	Our method of proof does not allow us to show H\"older continuity on all of $[0,T]$, because it is based on heat kernel estimates, which degenerate as $t \to 0$ (cf.~Equation \eqref{eq:change of probability estimate-bis}).
	However, we find it difficult to construct a counterexample.
\end{remark}

\begin{proof} 
  	The proof requires several steps and lies on the probabilistic interpretation of $ m $ 
	which we recall briefly here, and use in other parts of this paper.
	
	\emph{Step 1 (probabilistic interpretation):} 
	Consider the reflected diffusion process governed by 
	\begin{subequations}
	\begin{equation}\label{probabilistic-charact-1}
	\dd X_t = - b(t,X_t) \dd t + \sigma \dd W_t - \dd \xi_t^X, \quad X_0 \sim m_0,
	\end{equation}
	where $X_0$ is $\F_0$-measurable, 
	$(W_t)_{t\in[0,T]}$ is a $\mathbb F$-Wiener process that is independent of $X_0$, and set
	\begin{equation}\label{probabilistic-charact-2}
	\tau := \inf\{t\geq 0 : X_t \leq 0\} \wedge T.
	\end{equation}
	By virtue of the regularity assumptions on $b$, equation \eqref{probabilistic-charact-1} is well-posed 
	in the classical sense. 
	Furthermore, since the process $(\xi_t^X)_{t\geq 0}$ is monotone, $(X_t)_{t\in[0,T]}$ is a continuous semimartingale.  
	Hence, by means of It\^o's rule and the optional stopping theorem, we have
	for any test function $\phi \in \C_c^\infty([0,T) \times \overline{Q})$ satisfying \eqref{eq:weak fp:bc}:
	$$
	\E\left[\phi(0,X_0) \right] = \E \left[\int_0^\tau \left( -\phi_t(v,X_v)-\frac{\sigma^2}{2}\phi_{xx}(v,X_v)+\phi_x(v,X_v)b(v,X_v)  \right) \dd v \right],
	$$
	and thus the law of $X_t$ is a weak solution to the Fokker-Planck equation.
	The function $b$ being bounded, one sees that
	$$
	\E\left[ \int_0^T b(s,X_s)^2 \dd s \right] <\infty.
	$$
	Therefore, by virtue of the uniqueness for \eqref{eq:fp}-\eqref{eq:fp:bc} (cf. Proposition \ref{uniqueness:measure-valued}),
	 we obtain: 	
	\begin{equation}\label{probabilistic-charact-3}
	\int_{A} m(t,x)\dd x = \Pp(t< \tau; X_t\in A)
	\end{equation}
	for every Borel set $A\in \bar Q$ and for a.e. $t\in (0,T)$.
	\end{subequations}

	\emph{Step 2:} Now, let us  show that $t\to \Pp(t< \tau)$ is right continuous on $[0,T]$.
	In fact, we have for any $\e>0$ and $t\in[0,T]$
	\begin{subequations}
	\begin{eqnarray}\label{estimate::choice:of:ee-bis}
	\Pp(t< \tau)-\Pp(t+h< \tau) &=& \bb{P}(t+h \geq \tau; t < \tau)\\ \nonumber
	&\leq& \bb{P}(t+h \geq \tau; X_t \geq \e) + \bb{P}(t < \tau; X_t < \e).
	\end{eqnarray}
	On the one hand, for every $t\in[0,T]$ 
	\begin{equation}\label{estimate::choice:of:ee-bis:0}
	\lim_{\epsilon \to 0^+} \Pp(t<\tau; X_t < \epsilon) 
	\leq \lim_{\epsilon \to 0^+} \Pp\left(0<X_t <\e \right) =0, 
	\end{equation}
	thanks to the bounded convergence theorem.
	On the other hand
	\begin{multline*}
	\bb{P}(t+h \geq \tau; X_t \geq \e)
	\leq \bb{P}\left(\inf_{v \in [t,t+h]} X_v - X_t \leq -\e\right)\\
	\leq \bb{P}\left(\inf_{v \in [0,h]} \sigma (W_{t+v} - W_t) + (\xi_t^X-\xi_{t+h}^X) \leq -\e + h \|b\|_\infty  \right),
	\end{multline*}
	where we have used the fact that the local time is nondecreasing and $b$ is bounded. 
	Furthermore, by using \eqref{inequalities-growth-T}, it holds that
	$$
	\xi_t^X-\xi_{t+h}^X \geq \inf_{v \in [0,h]}(Y_t - Y_{t+v}) \geq \sigma \inf_{v \in [0,h]}(W_t - W_{t+v})- h \| b\|_{\infty} .
	$$
	Therefore
	\begin{equation*}
	\bb{P}(t+h \geq \tau; X_t \geq \e)
	\leq \bb{P}\left(\sup_{v \in [0,h]}  B_v -\inf_{v \in [0,h]}  B_v \geq  \frac{\e -2h \|b\|_\infty }{\sigma} \right),
	\end{equation*}
	where $(B_t)_{t\geq 0}$ is a Wiener process. 
	
	Now, choose $\e = \e(h) := h^{1/2}\log(1/h)$. We have $\e(h) \to 0$ as $h \to 0^+$, and
	by using Markov's inequality and the distribution of the maximum of Brownian motion we get:
	\begin{equation}\label{estimate::choice:of:ee}
	\bb{P}(t+h \geq \tau; X_t \geq \e)
	\leq   \frac{2\sigma}{\e(h) -2h \|b\|_\infty} \E\left| B_h \right| \leq \frac{2 \sigma }{\log(1/h) -2\|b\|_\infty h^{1/2}}.
	\end{equation}
	Thus $0 \leq \Pp(t< \tau)-\Pp(t+h< \tau) \to 0$ as $h \to 0^+$. 
	\end{subequations}

	 \emph{Step 3 (H\"older estimates):}
	Now, we prove \eqref{eq:eta holder}-\eqref{eq: holder-dual holder}. 
	At first, note that \eqref{probabilistic-charact-3} entails
	\begin{equation}\label{probabilistic-charact-strong}
	 \int_0^L \phi(x) m(t,x)\dd x = \E\left[ \phi(X_{t})\1_{t<\tau} \right]
	\end{equation}
	for a.e. $t\in (0,T)$ and for any $\phi \in \C(\bar Q)$. 
	Actually \eqref{probabilistic-charact-strong} holds for every $t\in[0,T]$, 
	since the RHS  and LHS of \eqref{probabilistic-charact-strong} are both right continuous on $[0,T]$, 
	and $m_0$ is supported on $(0,L]$.
	Indeed, on the one hand $t\to \int_0^L \phi(x) m(t,x)\dd x$ 
	is continuous on $[0,T]$ for any continuous function $\phi$ on $\bar Q$, since $m \in \C([0,T]; L^1(Q))$
	(cf. \cite[Theorem 3.6]{porretta2015weak}).
	On the other hand, for any $\phi \in \C(\bar Q)$
	\begin{multline}\label{probability-continuity-step3}
	\E\left| \phi(X_{t+h})\1_{t+h<\tau} -\phi(X_{t})\1_{t<\tau}  \right|\\
	 \leq \| \phi\|_{\infty}(\Pp(t< \tau)-\Pp(t+h< \tau) )
	+ \E\left| \phi(X_{t+h}) -\phi(X_{t})  \right|, 
	\end{multline}
	so that 
	$$
	\lim_{h\to 0^+}\E\left| \phi(X_{t+h})\1_{t+h<\tau} -\phi_\e(X_{t})\1_{t<\tau}  \right|=0
	$$
	thanks to \eqref{estimate::choice:of:ee-bis}-\eqref{estimate::choice:of:ee}, and the bounded convergence theorem.

	Now, let us fix $\e>0$ and define $\phi_\e = \phi_\e(x)$ to be a smooth cut-off function on $[0,L]$, 
	which satisfies the following conditions:
	\begin{equation}\label{cutt-off-function:specifications}
	 0\leq \phi_\e \leq 1; \ \ 0 \leq \phi'_\e \leq 2/\e; \ \ \phi_\e \1_{[0,\e]}=0; \ \ \phi_\e \1_{[2\e,L]}=1.
	\end{equation}
	As a first step, we aim to derive an estimation of the concentration of mass at the origine. Namely, we want to 
	show that for an arbitrary $k >1$, 
	\begin{equation} \label{eq:boundary probability}
	\int_0^L (1-\phi_\e(x)) m(t,x)\dd x \leq C(k,\|b\|_\infty) \left(1-e^{-\pi^2 t/4L^2} \right)^{-1/2k}\e^{1/2k} \quad \forall t\in (0,T].
	\end{equation}
	Given \eqref{probabilistic-charact-strong}, this is equivalent to showing that
	\begin{equation} \label{eq:boundary probability 2}
	 \E\left[ (1-\phi_\e(X_{t}))\1_{t<\tau} \right] \leq C(k,\|b\|_\infty) \left(1-e^{-\pi^2  t/4L^2}\right)^{-1/2k}\e^{1/2k} 
	 \quad \forall t\in (0,T]
	\end{equation}
	holds for any $k>1$.
	Apply Girsanov's Theorem with the following change of measure:
	$$
	\left. \frac{\dd \Qq}{\dd \Pp}\right|_{\F_t} = \exp\left\{ - \sigma^{-1} \int_0^t b(s,X_s) \dd W_s 
	- \frac{\sigma^{-2}}{2} \int_0^t  b(s,X_s)^2 \dd s  \right\} 
	=: \Psi_t.
	$$
	Under $\Qq$, the process $(X_t)_{t\in[0,T]}$ is a reflected Brownian motion at $L$, 
	with initial condition $X_0$, thanks to \eqref{skorokhod-problem-equation-formula}.
	Moreover, by virtue of H\"older inequality, we have for every $k>1$:
	\begin{multline*}
	\E_{\Pp}\left[ (1-\phi_\e(X_{t}))\1_{t<\tau} \right]\\
	= \E_{\Qq}\left[ \Psi_t^{-1} (1-\phi_\e(X_{t}))\1_{t<\tau} \right]
	 \leq \E_{\Qq}[\Psi_t^{-k'}]^{1/k'} 
	\E_{\Qq}\left[  (1-\phi_\e(X_{t}))^k \1_{t<\tau} \right]^{1/k} \\
	\leq \E_{\Pp}[\Psi_t^{1-k'}]^{1/k'} \E_{\Qq}\left[  (1-\phi_\e(X_{t}))^k \1_{t<\tau} \right]^{1/k} \\
	\leq \E_{\Pp}\left[\exp\left\{ C(k,\sigma) \int_0^t b(s,X_s)^2 \dd s  \right\}\right]^{1/2k'} 
	\E_{\Qq}\left[  (1-\phi_\e(X_{t}))^k \1_{t<\tau} \right]^{1/k}.
	\end{multline*}
	Indeed, one checks that
	$$
	\E_{\Pp}[\Psi_t^{1-k'}]^{1/k'} \leq \E_{\Pp}\left[  Z_t \right]^{1/2k'} \E_{\Pp}\left[  \exp\left\{ 2\left(\frac{1-k'}{\sigma} \right)^2 \int_0^t b(s,X_s)^2 \dd s \right\} \right]^{1/2k'},
	$$
	where $(Z_{t})_{t\geq 0}$ is a super-martingale.
	Hence, using the fact that $b$ is bounded, we obtain
	\begin{equation} \label{eq:change of probability estimate}
	\E_{\Pp}\left[ (1-\phi_\e(X_{t}))\1_{t<\tau} \right] \leq 
	C(k,\|b\|_\infty)\E_{\Qq}\left[  (1-\phi_\e(X_{t}))^k \1_{t<\tau} \right]^{1/k}.
	\end{equation}	
	Now 
	$$
	\E_{\Qq}\left[  (1-\phi_\e(X_{t}))^k \1_{t<\tau} \right] = \int_0^L (1-\phi_\e(x))^k w(t,x)\dd x,
	$$
	where $w$ solves
	$$
	w_t = \frac{1}{2}w_{xx}, \ w(t,0) = 0, \ w_x(t,L) = 0, w|_{t=0} = m_0.
	$$
	We can compute $w$ via Fourier series, namely
	$$
	w(t,x) = \sum_{n\geq 1} A_n e^{-\lambda_n^2 t/2} \sin (\lambda_n x) , \quad  
	A_n := \frac{2}{L}\int_0^L \sin (\lambda_n y) \dd m_0(y), \quad \lambda_n := \frac{(2n-1)\pi}{2L}.
	$$
	Note that
	\begin{align*}
	\int_0^L (1-\phi_\e(x))^k w(t,x)\dd x
	&\leq (2\e)^{1/2} \|w(t,\cdot)\|_{L^2}\\
	&\leq \e^{1/2} \left(\sum_{n\geq 1} L |A_n|^2  e^{-\lambda_n^2  t}\right)^{1/2} \ \text{(Parseval)}\\
	&\leq \e^{1/2} \left(\frac{4}{L(1-e^{-\pi^2  t/4L^2})} \right)^{1/2}.
	\end{align*}
	So \eqref{eq:change of probability estimate} now yields
	\begin{equation}\label{eq:change of probability estimate-bis}
	\E_{\Pp}\left[ (1-\phi_\e(X_{t}))\1_{t<\tau} \right] \leq C(k,\|b\|_\infty) 
	\left(\frac{4}{L(1-e^{-\pi^2  t/4L^2 })} \right)^{1/2k}\e^{1/2k}
	\end{equation}
	which is \eqref{eq:boundary probability 2}.
	This in turn implies \eqref{eq:boundary probability}.
	
	
	Furthermore, note that for any $1<s<3/2$, 
	\begin{equation} \label{eq:m holder}
	\|m(t_1)-m(t_2)\|_{W^{-1}_s} \leq C\left(\|m_0\|_{L^1},\|m|b|^2\|_{L^1}\right) |t_1-t_2|^{1-1/s} \quad \forall t_1,t_2 \in [0,T],	
	\end{equation} 
	where $W^{-1}_{s}(Q)$ is the dual space of  $W_{0}^{1,s'}(Q) :=\left\{   v \in W_{s'}^1(Q)\ :\  v(0)=0 \right\}$.
	This claim follows from \cite[Proposition 3.10(iii)]{porretta2015weak}, where we obtain the estimate
	\begin{multline}
	\label{porretta-estimates}
	\|m\|_{L^\infty(0,T;L^1(Q))} + \|\nabla m\|_{L^s(Q_T)} + \|m\|_{L^v(Q_T)} + \|m_t\|_{L^s(0,T;W_s^{-1}(Q))}\\
	 \leq C\left(\|m_0\|_{L^1},\|m|b|^2\|_{L^1}\right)
	\end{multline}
	for any $s$ up to $3/2$ and $v$ up to 3.
	In particular, \eqref{eq:m holder} follows from the estimate on
	 $\|m_t\|_{L^s(0,T;W_s^{-1}(Q))}$. 
	 Now, fix $0 < t_1 \leq t_2 \leq T$, 	
	and let $\phi_\e$ be the cut-off function that is defined in \eqref{cutt-off-function:specifications}.
	Based on the specifications of \eqref{cutt-off-function:specifications}, observe that 
	$$
	\|\phi_\e\|_{W_{s'}^{1}} \leq C\e^{-1/s}.
	$$
	Since $\phi_\e$ satisfies Neumann boundary conditions at $x = L$ and Dirichlet at $x = 0$, 
	it is a valid test function and we can 		
	appeal to the estimates above to obtain for any $k > 1$,
	\begin{multline}\label{eq:boundary probability: holder}
	\left|\eta(t_1)-\eta(t_2)\right|  
	= \left| \int_{0}^{L} \left\{  (1-\phi_\e(x)) + \phi_\e(x)  \right\} (m(t_1,x) -  m(t_2,x))\dd x \right| \\
	\leq \int_0^{L} |1-\phi_\e(x)||m(t_1,x) -  m(t_2,x)|\dd x
	+ \left|\int_0^L \phi_\e(x)(m(t_1,x)-m(t_2,x))\dd x\right|\\
	\leq \int_0^{L} (1-\phi_\e(x))(m(t_1,x) +  m(t_2,x))\dd x
	+ \left|\int_0^L \phi_\e(x)(m(t_1,x)-m(t_2,x))\dd x\right|\\
	\leq C(k,\|b\|_\infty) \left(1-e^{-\pi^2\ t_1/4L^2} \right)^{-1/2k}\e^{1/2k} 
	+ \|\phi_\e\|_{W_{s'}^{1}}\left\|m(t_1)-m(t_2) \right\|_{W_s^{-1}}	\\
	\leq C(k,\|b\|_\infty) \left(1-e^{-\pi^2 t_1/4L^2} \right)^{-1/2k}\e^{1/2k} + C \e^{-1/s}|t_1-t_2|^{1-1/s},
	\end{multline}
	where we have used \eqref{eq:boundary probability} in the penultimate line and \eqref{eq:m holder} in the ultimate line.
	Given $0 < \gamma < (s-1)/(s+2)$, we take $\e = |t_1-t_2|^{s(1-\g)-1}$ and then set 
	$k = \frac{s(1-\gamma)-1}{2\gamma} > 1$ to obtain \eqref{eq:eta holder}. 
	
	Finally, let $\phi \in \C^{\a}(\bar Q)$ for some $\a >0$, an let $t_0 \in (0,T)$. In view of \eqref{probabilistic-charact-strong}, 
	we have for every $t_1,t_2 \in [t_0,T]$,  
	\begin{multline*}
	\left| \int_0^L \phi(x) (m(t_1,x)-m(t_2,x)) \dd x   \right| \leq 
	\E \left| \phi(X_{t_1}) \1_{t_1 < \tau} -  \phi(X_{t_2}) \1_{t_2 < \tau} \right| \\
	\leq \| \phi \|_{\C^{\a}} \left(  \left| \eta(t_1)-\eta(t_2)  \right| + \E\left| X_{t_1}-X_{t_2}  \right|^\a   \right).
	\end{multline*}
	Hence, by using \eqref{eq:eta holder} and the Burkholder-Davis-Gundy inequality 
	\cite[Thm IV.42.1]{Rogers2000martingales}, we deduce the desired result:
	$$
	\left| \int_0^L \phi(x) (m(t_1,x)-m(t_2,x)) \dd x   \right| \leq C(t_0, \|b\|_{\infty})\| \phi \|_{\C^{\a}} \left| t_1-t_2\right|^{\b},
	$$
	for some $\b>0$.
	
	 \emph{Step 4 (general data):} Now, we suppose that $m_0$ is a probability measure satisfying \eqref{Aa1}, 
	 and not necessarily an element of $L^1(Q)$.
	 Let us  choose a  sequence $(m_{0}^{n})\subset L^1(Q)_+$,
	 which converges weakly (in the sense of measures) to $m_0$, such that 
	 \begin{equation}\label{regularized-sequence}
	 \left\| m_{0}^{n} \right\|_{L^1} \leq \int_0^L \dd m_0 \leq 1,
	  \end{equation}
	 and let $m^n$ to be the weak solution to \eqref{eq:fp}-\eqref{eq:fp:bc} starting from $m_0^n$. 
	 The function $b$ being bounded, we can use \cite[Proposition 3.10]{porretta2015weak} to extract a 
	 subsequence of $(m^n)$, which converges to $m$  in $L^1(Q_T)$. 
	 Owing to \eqref{eq:eta holder}, the sequence  $\eta^n:=\eta(m^n)$ is equicontinuous. 
	 Hence, one can extract further a subsequence to deduce that $\eta$ is  continuous on $(0,T]$. 
	 Combining this conclusion with the fact that $t \to \Pp(t<\tau)$ is right continuous on $[0,T]$
	 and \eqref{probabilistic-charact-3}, we deduce in particular that
	 \begin{equation}\label{probabilistic-charact-e.w}
	 \eta(t)= \Pp(t<\tau), \quad \forall t\in(0,T].
	 \end{equation}
	 Now, since $m_0$ is supported on $(0,L]$ one has $ \eta(m_0)=\eta(0)=\Pp(0<\tau)=1$, 
	 which in turn entails that $\eta$ is continuous on $[0,T]$ thanks to 
	 \eqref{estimate::choice:of:ee-bis}-\eqref{estimate::choice:of:ee} and \eqref{probabilistic-charact-e.w}. 
	 The proof is complete. 
\end{proof}
\begin{remark}\label{probabilistic-measures-every time}
When $m_0$ satisfies \eqref{Aa1} and does not necessarily belong to $L^1(Q)$, the probabilistic characterisation \eqref{probabilistic-charact-strong} still holds for every $t\in [0,T]$. In fact, using the same approximation techniques as in Lemma \ref{continuity of measure}- Step 4, and appealing to \eqref{eq: holder-dual holder} and \eqref{probabilistic-charact-3}, it holds that
\begin{equation*}
	 \int_0^L \phi(x) m(t,x)\dd x = \E\left[ \phi(X_{t})\1_{t<\tau} \right]
	\end{equation*}
	for every $t\in [0,T]$, $\a>0$ and $\phi \in \C^\a (\bar Q)$. 
	Thus, \eqref{probabilistic-charact-strong} ensues by using density arguments.
\end{remark}

\subsection{A priori estimates}\label{section2: A priori estimates}

Now, we collect several a priori estimates for system \eqref{MFG-limit-1}.

\begin{lemma} \label{holder-estimates}
	Suppose that $(u,m)$ satisfies the system \eqref{MFG-limit-1} such that  $m \in L^1(Q_T)_+$,
	and $u$  belongs to $W_s^{1,2}(Q_T)$ for large enough $s>1$.
	Then, we have:
	\begin{enumerate}[ label=(\roman*)]
	\item the maps $u$ and $ u_{x}$ are non-negative; in particular
	\begin{equation}\label{uniform-bound::G-lemma23}
	0 \leq q_{u,m} \leq 1/2;
	\end{equation}
	\item  there exists $\theta>0$ and a constant $c_0>0$ such that
	\begin{equation}\label{Holder estimates-1}
	\|u\|_{C^\theta(\overline{Q_T})}, \|u_x\|_{C^\theta(\overline{Q_T})} \leq c_0
	\end{equation}
	where $c_0$ depends only on $T$ and data.
	In addition, we have
	$$
 	\|u_{xx}\|_{C^\theta({Q'})}, \leq c_1(Q',\theta) \ \ \forall Q' \subset \subset (0,T)\times (0,L];
	$$
	\end{enumerate}

	If in addition $m_0$ belongs to  $L^1(Q)$, then there exists a H\"older exponent $\theta>0$ such that 
	$$
	\|p_{u,m}\|_{C^\theta([t_0,T] \times [0,L])} \leq c_2(t_0, \theta), \quad \forall t_0 \in (0,T),
	$$
	and
	$$
	\|u_{t}\|_{C^\theta({Q'})} \leq c_2(Q',\theta) \ \ \forall Q' \subset \subset (0,T)\times (0,L].
	$$
\end{lemma}

\begin{proof}
	For large enough $s>1$, we know that $u,u_x \in \C(\overline{Q_T})$ thanks to Sobolev-H\"older embeddings.
	In view of
	$$
	-u_t-\frac{\sigma^2}{2}u_{xx}+ru\geq 0,
	$$
	one easily deduces that $u\geq e^{-rT}\min_{x} u_{T}$, which entails in particular that $u\geq 0$ thanks to \eqref{A1}. 
	Thus, the minimum is attained at $u(t,0)=0$, so that $u_{x}(t,0)\geq 0$ for all $t\in[0,T]$.
	Differentiating the first equation in \eqref{MFG-limit-1} we have that $u_x$ is a generalised solution 
	(cf. \cite[Chapter III]{Parabolic67}) of the following parabolic equation:
	$$
	u_{xt}+\frac{\sigma^2}{2}u_{xxx}-ru_{x}-q_{u,m}u_{xx}=0.
	$$
	By virtue of the maximum principle \cite[Theorem III.7.1]{Parabolic67}  we infer that $u_{x}\geq 0$, 
	since $u_{x}(t,0)$, $u_{x}(t,L)$ and $u'_{T}$ are all non-negative functions. 
	Therefore \eqref{uniform-bound::G-lemma23} follows straightforwardly from \eqref{MFG-quantity-def} thanks to  
	Lemma \ref{fixed::point::market-price}.

	Note that $u$ solves a parabolic equation with bounded coefficients. 
	Since compatibility conditions of order zero are fulfilled thanks to \eqref{A1}, then from \cite[Theorem IV.9.1]{Parabolic67} 
	we have an estimate on $u$ in $W_k^{1,2}(Q_{T})$ for arbitrary $k>1$, namely
	\begin{equation}
	\label{eq:arbitrary k est}
	\|u\|_{W_k^{1,2}(Q_{T})} \leq C\left(\|q_{u,m}\|_{L^k(Q_T)} + \|u_T\|_{W_{k}^{2-\frac{2}{k}}(Q_{T})} \right)
	\leq C\left(\|q_{u,m}\|_{L^\infty(Q_T)} + \|u_T\|_{W_{k}^{2-\frac{2}{k}}(Q_{T})} \right).
	\end{equation}
	This	estimate depends only on $T$, $k$ and data, thanks to \eqref{uniform-bound::G-lemma23}. 
	We deduce \eqref{Holder estimates-1} thanks to Sobolev-H\"older embeddings.
		
	Now, let $\phi \in \C_c^\infty((0,T)\times (0,+\infty))$.
	Observe that $w = \phi u_x$ satisfies
	$$
	w_t + \frac{\sigma^2}{2} w_{xx} - rw - q_{u,m}w_{x} = \phi_t u_x + \sigma^2 \phi_x u_{xx} 
	+ \frac{\sigma^2}{2} \phi_{xx}u_x - q_{u,m}\phi_x u_{x}.
	$$
	For any $k>1$, the right-hand side is bounded in $L^k(Q_T)$ with a constant that depends only on 
	$\phi$, and previous estimates. 
	Since $w$ has homogeneous boundary conditions, we deduce from \cite[Theorem IV.9.1]{Parabolic67} 
	that $\|w_x\|_{C^\t(\overline{Q_T})}$ is bounded by a constant depending only on the norm of $\phi$ 
	and previous estimates.
	The local H\"older estimate on $u_{xx}$ then follows.

	Let $p(t,x) = p_{u,m}(t,x)$.
	Recall that $p(t,x) = f(\ell(m,p)(t),u_x(t,x))$ where $f(x,y) := x - \frac{1}{2}(x-y)^+$ (cf. Lemma \ref{fixed::point::market-price}).
	Since $f$ is 1-Lipschitz in the first variable and $\frac{1}{2}$-Lipschitz in the second, we deduce that
	\begin{equation} 
	\label{eq:p holder}
	|p(t_1,x_1)-p(t_2,x_2)| \leq |\ell(m,p)(t_1)-\ell(m,p)(t_2)| + \frac{1}{2}|u_x(t_1,x_1)-u_x(t_2,x_2)|.
	\end{equation}
	In particular, for each $t$,
	\begin{equation}
	\label{eq:p holder in x}
	|p(t,x_1)-p(t,x_2)| \leq \frac{1}{2}|u_x(t,x_1)-u_x(t,x_2)|
	\end{equation}
	which, by \eqref{Holder estimates-1}, implies that $p(t,\cdot)$ is H\"older
	continuous for every $t$.
	
	Now, we further assume that $m_0\in L^1(Q)_+$ to use \eqref{eq:eta holder}-\eqref{eq: holder-dual holder}. 
	We shall use the following function which is introduced in Lemma \ref{fixed::point::market-price}:
	$$
	\ell(m,p)(t)= a\left( m(t) \right)+c\left( m(t) \right)\overline{p}(t), \quad \mbox{ where }\quad
\overline{p}(t)=\frac{1}{\eta(m(t))}\int_0^{L} p(t,x) m(t,x) \dd x.
	$$
	Fix $t_0 \in (0,T)$ and for $t_1,t_2$ in $[t_0,T]$ write
	\begin{multline}\label{decomposition:holde:price}
	\ell(m,p)(t_1)-\ell(m,p)(t_2)
	= a(m(t_1)) - a(m(t_2)) \\
	+ \k(a(m(t_1)) - a(m(t_2)))\int_0^L p(t_1,.) \dd m(t_1)
	\\
	+\k a(m(t_2))\int_0^L p(t_1,.)\dd ( m(t_1)-  m(t_2))
	\\
	+ \k a(m(t_2))\int_0^L (p(t_1,.)-p(t_2,.))\dd m(t_2),
	\end{multline}
	where we have used the fact that $c(m) = \k a(m)\eta(m)$.
	Observe that $\eta \to \frac{1}{1+\k \eta}$ is $\k$-Lipschitz in the $\eta$ variable, and recall that $p(t_1,\cdot)$ is H\"older
	continuous.
	Moreover, by virtue of \eqref{Holder estimates-1} we know that $q_{u,m}$ satisfies  \eqref{assumption-SDEs}.
	Therefore, using the upper bound on $a(m), c(m)$ and \eqref{eq:eta holder}-\eqref{eq: holder-dual holder}  we infer that
	\begin{equation} \label{eq:ell holder}
	|\ell(m,p)(t_1)-\ell(m,p)(t_2)| \leq C|t_1-t_2|^\b + \frac{\k}{1+\k}\|p(t_1,\cdot)-p(t_2,\cdot)\|_\infty.
	\end{equation}
	Note that the constant in \eqref{eq:ell holder} depend only on $c_0$ and $\k$ 
	thanks to \eqref{uniform-bound::G-lemma23}, \eqref{Holder estimates-1}
	and Lemma \ref{fixed::point::market-price}.
	Using now \eqref{eq:ell holder} in \eqref{eq:p holder}, and choosing $\t$ small enough, we deduce
	\begin{equation}
	\label{eq:p holder in t}
	\frac{1}{1+\k}\|p(t_1,\cdot)-p(t_2,\cdot)\|_\infty \leq C|t_1-t_2|^\b + \frac{1}{2}\|u_x(t_1,\cdot)-u_x(t_2,\cdot)\|_\infty \leq C|t_1-t_2|^\t.
	\end{equation}
	Putting together \eqref{eq:p holder in x} and \eqref{eq:p holder in t} we infer that $p$ has a H\"older estimate, 
	whereupon by \eqref{eq:ell holder} so does $\ell(m,p)$.
	Thus $q_{u,m}$ also has a H\"older estimate, and so does $u_t$ by the HJB equation satisfied by $u$.
\end{proof}

\subsection{Well-posedness} \label{section 2: Well-posedness}
We are now in position to prove the main result of this section.

\begin{proof}[Proof of Theorem \ref{theorem:wellposedness}]
The proof requires several steps, the key arguments being precisely the estimates collected in Lemmas \ref{fixed::point::market-price}-\ref{holder-estimates}.

\emph{Step 1 (data in $L^1$):}
We suppose that $m_{0}$ is an element of $L^1(Q)$ satisfying \eqref{Aa1}. Define $\X$ to be the space of couples $(\varphi,\nu)$, such that  $\varphi$ and $\varphi_x$ are globally continuous on $\overline{Q_T}$, and  $\nu$ belongs to $L^{1}(Q_T)_+$. The functional space  $\X$  endowed with the norm: 
$$
\|(\varphi,\nu) \|_{\X} := \| \varphi\|_{\infty}+\| \varphi_x\|_{\infty} + \| \nu \|_{L^1}
$$
is a Banach space.
Consider the map $\T: (\varphi,\nu, \la) \in \X\times [0,1] \to (w,\mu)$ where $(w,\mu)$ are given by the following parametrized system of coupled partial differential equations:
\begin{equation}
\label{equation1-Th26}
\left\{
\begin{aligned}
&(i) \quad w_{t}+\frac{\sigma^{2}}{2} w_{xx}-r w + \la^2 q_{\varphi,\nu}^2=0
 \quad \mbox{ in }Q_T,\\
&(ii)\quad \mu_{t}-\frac{\sigma^{2}}{2}\mu_{xx}-\left\{ \la q_{\varphi,\nu} \mu \right\}_{x}=0 \quad \mbox{ in }Q_T,\\
&(iii)\quad \mu(t,0)=0,\quad  w(t,0)=0,\quad  w_{x}(t,L)=0\quad \mbox{ in } [0,T],  \\
&(iv)\quad \mu(0)=\la m_{0}, \quad  w(T,x) = \la u_{T}(x)\quad \mbox{ in } [0,L],\\
&(v)\quad \frac{\sigma^{2}}{2}\mu_{x} + \la q_{\varphi,\nu}\mu=0    \quad \mbox{ in } [0,T]\times \{ L\}.
\end{aligned}
\right.
\end{equation} 
By virtue of Lemma \ref{fixed::point::market-price},  the map $q_{\varphi,\nu}$ is well-defined for any $(\varphi,\nu)\in \X$, and satisfies
\begin{equation}\label{bound1-Th26}
\left| q_{\varphi,\nu}  \right| \leq C(1+\| \varphi_x\|_\infty).
\end{equation}
In view of \cite[Theorem IV.9.1]{Parabolic67}, the function $w$ exists and is bounded in $W_s^{1,2}(Q_T)$ for any $s>1$, by a constant which depends on $\| \varphi_x \|_{\infty}$ and data.
(Note that the required compatibility conditions hold owing to \eqref{A1}.
Although \cite[Theorem IV.9.1]{Parabolic67} is stated for Dirichlet boundary conditions, its proof is readily adapted to Neumann or mixed boundary conditions as in the present context; cf.~the discussion in the first paragraph of \cite[Section IV.9]{Parabolic67}).
We deduce that 
$$
\| w \|_{\C^\a}+\| w_x\|_{\C^\a} \leq C(T, L, u_T, \| \varphi_x\|_{\infty})
$$
for some $\a>0$. On the other hand, it is well known (see e.g.~\cite[Chapter III]{Parabolic67})
 that for any $(\varphi,\nu)\in \X$, equation \eqref{equation1-Th26}(ii) has a unique weak solution $\mu$. 
Therefore, $\T$ is well-defined. 	\\
 	Let us now prove that $\T$ is continuous and compact.
 	Suppose $(\varphi_n,\nu_n,\lambda_n)$ is a a bounded sequence in $\X \times [0,1]$ and let $(w_n,\mu_n) = \T(\varphi_n,\nu_n,\lambda_n)$.
 	To prove compactness, we show that, up to a subsequence, $(w_n,\mu_n)$ converges to some $(w,\mu)$ in $\X$.
 	Since $(\varphi_n)_x$ is uniformly bounded, by virtue of \cite[Proposition 3.10]{porretta2015weak}, the sequence $\mu_n$ is relatively compact in $L^1(Q_T)_+$, thanks to \eqref{bound1-Th26} (cf. \eqref{eta-decreasing::theorem24} below  where more details are given).
 	Since $w_n$ and $(w_n)_x$ are uniformly bounded in $\C^{\a}(\overline{Q}_{T})$, by the Ascoli-Arzel\`a Theorem and uniform convergence of the derivative there exists some $w$ such that $w,w_x$ are continuous in $\overline{Q_T}$ and, passing to a subsequence, $w_n \to w$ and $(w_n)_x \to w_x$ uniformly, where in fact $w_n \rightharpoonup w$ weakly in $W_s^{1,2}(Q_T)$ for any $s>1$.
 	This is what we wanted to show.\\ 	
 	To prove continuity, we assume $(\varphi_n,\nu_n,\lambda_n) \to (\varphi,\nu,\lambda)$ in $\X \times [0,1]$.
 	It is enough to show that, after passing to a subsequence, $\T(\varphi_n,\nu_n,\lambda_n) \to \T(\varphi,\nu,\lambda)$.
 	By the preceding argument, we can assume $\T(\varphi_n,\nu_n,\lambda_n) \to (w,\mu)$.
 	We can also use estimates \eqref{eq:F p}-\eqref{eq:F m} to deduce that $q_{\varphi_n,\nu_n} \to q_{\varphi,\nu}$ a.e.~(cf.~the proof of Equation \eqref{eq:Q converges} below), and since $q_{\varphi_n,\nu_n}$ is uniformly bounded we can also assert $q_{\varphi_n,\nu_n} \to q_{\varphi,\nu}$ in $L^s$ for any $s \geq 1$.
 	Then we deduce that $(w,\mu)$ is a solution of \eqref{equation1-Th26} for the given $(\varphi,\nu,\lambda)$.
 	Therefore, $(w,\mu) = \T(\varphi,\nu,\lambda)$, as desired.

Now, let $(u,m)\in \X$ and $\la \in [0,1]$ so that $(u,m)=\T(u,m,\la)$. Then $(u,m)$ satisfies assumptions of Lemma \ref{holder-estimates}  with $m_0, u_T, q_{u,m}$ replaced by $\la m_0, \la u_T$ and $\la q_{u,m}$, respectively. Since the bounds of Lemma \ref{holder-estimates} carry through uniformly in $\la \in [0,1]$ we infer that
$$
\| (u,m)\|_{\X} \leq 1\vee c_0,
$$ 
where $c_0>0$ is the constant of Lemma \ref{holder-estimates}. In addition, for $\la=0$ we have $\T(u,m,0)=(0,0)$.
Therefore, by virtue of Leray-Schauder fixed point Theorem (see e.g. \cite[Theorem 11.6]{Gilbarg}),
we deduce the existence of a solution $(u,m)$ in $\X$ to system \eqref{MFG-limit-1}.

\emph{Step 2 (measure data):} We deal now with general $m_0$, i.e. a probability measure  that is supported on $(0,L]$.  Let $(m_{0}^{n})\subset L^1(Q)_+$ be a  sequence of functions, which converges weakly (in the sense of measures) to $m_0$, and such that 
$$
\left\| m_{0}^{n} \right\|_{L^1} \leq \int_0^L \dd m_0 \leq 1, \
\mbox{ and } \
{\rm supp}(m_0^n) \subset (0,L].
$$
For any $n\geq 1$, define $(u^n, m^n)$ to be a solution in $\X$ to system \eqref{MFG-limit-1} starting from $m_0^n$.

In view of \cite[Proposition 3.10 (iii)]{porretta2015weak} and \eqref{uniform-bound::G-lemma23}, the corresponding solutions $m^n$ to the non-local Fokker-Planck equation lie in a relatively compact set of $L^1(Q_{T})$. Moreover, it holds that
\begin{equation}\label{eta-decreasing::theorem24}
m^{n}\geq 0 \ \mbox{ and } \ \sup_{0\leq t\leq T}\left\| m^{n}(t) \right\|_{L^1}\leq \int_0^L \dd m_0.
\end{equation}
Passing to a subsequence we have $m^{n} \to m$ in $L^1(Q_T)$, $m^{n}(t) \to m(t)$ in $L^1(Q)$ for a.e.~$t$ in $(0,T)$, 
and $m^{n} \to m$ for a.e.~$(t,x)$ in $Q_T$. 
It follows that $m \in L^1(Q_T)_+$  and 
\begin{equation}\label{measure-bound-main-theorem}
\left\| m(t) \right\|_{L^1}\leq 1 \quad \mbox{ for a.e. } t\in (0,T).
\end{equation}
In addition, we know that $q_{u,m}$ fulfils the assumptions of Lemma \ref{continuity of measure}.
Thus $t\to \left\| m(t) \right\|_{L^1}$ is continuous on $(0,T]$, so that \eqref{measure-bound-main-theorem} holds for avery $t\in(0,T]$.
Furthermore, we can appeal to
the probabilistic characterisation  \eqref{probabilistic-charact-strong}, thanks to Remark \ref{probabilistic-measures-every time}, to get
\begin{multline*}
\left| \int_0^L \phi(x) (m(t+h,x)-m(t))\dd x   \right| \leq \E\left| \phi(X_{t+h})\1_{t+h<\tau} -   \phi(X_{t})\1_{t<\tau} \right| \\
\leq \| \phi \|_{\infty} |\eta(t)-\eta(t+h) | + \E\left| \phi(X_{t+h}) -   \phi(X_{t})\right|
\end{multline*}
for every $\phi \in \C(\bar Q)$, and $t\in[0,T]$. Now owing to Lemma \ref{continuity of measure}, $\eta$ is continuous on $[0,T]$. Hence, by taking the limit in the last estimation we infer that
$$
\lim_{h\to 0}\int_0^L \phi(x) (m(t+h,x)-m(t))\dd x=0
$$
thanks to the bounded convergence theorem.
Consequently the map $t\to m(t)$ is continuous on $[0,T]$ with respect to the strong topology of $\M(\bar Q)$.

On the other hand, by Lemma \ref{holder-estimates} we have that $u^{n}$, $u_x^{n}$ are uniformly bounded in $\C^{\t}(\overline{Q}_{T})$, and $u_t^{n}$, $u_{xx}^{n}$  are uniformly bounded in $\C^{\t}(Q')$ for each $Q' \subset \subset (0,T)\times (0,L]$.
Thus, up to a subsequence we obtain that $u,u_x \in \C(\overline{Q_T})$, and
\begin{equation}\label{u-regularity}
u^n \to u \in \C^{1,2}((0,T)\times (0,L])
\end{equation}
where the convergence is in the $\C^{1,2}$ norm on arbitrary compact subsets of $(0,T)\times (0,L]$.

To show that the Hamilton-Jacobi equation holds in a classical sense and the Fokker-Planck equation holds in the sense of distributions, it remains to show that
\begin{equation} \label{eq:Q converges}
q_{u^n,m^n} \to q_{u,m} \ \mathrm{a.e.}
\end{equation}
at least on a subsequence.
Set $p^n = p_{u^n,m^n} = \La(u_x^n,m^n,p^n)$ and $p = p_{u,m} = \La(u_x,m,p)$, with $\La$ defined in Lemma \ref{fixed::point::market-price}.
Using \eqref{eq:F p}-\eqref{eq:F m} we get
\begin{multline}
\|p^n(t,\cdot)- p(t,\cdot)\|_\infty \leq \|\La(u_x^n,m^n,p^n)(t,\cdot)-\La(u_x,m^n,p^n)(t,\cdot)\|_\infty \\
+ \|\La(u_x,m^n,p^n)(t,\cdot) - \La(u_x,m^n,p)(t,\cdot)\|_\infty
+ \|\La(u_x,m^n,p)(t,\cdot) - \La(u_x,m,p)(t,\cdot)\|_\infty\\
\leq \frac{1}{2}\|u_x^n - u_x\|_\infty + \frac{\k}{1+\k}\|p^n(t,\cdot)- p(t,\cdot)\|_\infty + |\ell(m^n,p)(t) - \ell(m,p)(t)|
\end{multline}
which means
\begin{equation}
\|p^n(t,\cdot)- p(t,\cdot)\|_\infty \leq \frac{1+\k}{2}\|u_x^n - u_x\|_\infty  + (1+\k)|\ell(m^n,p)(t) - \ell(m,p)(t)|.
\end{equation}
Noting that (up to a subsequence) $m^n(t) \to m(t)$ in $L^1(Q)$ a.e., we use the fact that $a(m),c(m),\eta(m)$ are all continuous with respect to this metric to deduce that
\begin{equation} \label{eq:ell m to 0}
|\ell(m^n,p)(t) - \ell(m,p)(t)| \to 0 \ \ \mathrm{a.e.} \ t \in (0,T)
\end{equation}
from which we conclude that 
\begin{equation} \label{eq:p to 0}
\|p^n(t,\cdot)- p(t,\cdot)\|_\infty \to 0 \ \ \mathrm{a.e.} \ t \in (0,T).
\end{equation}
Now from \eqref{eq:p to 0} and \eqref{eq:ell p} we have
\begin{equation} \label{eq:ell p to 0}
|\ell(m,p^n)(t) - \ell(m,p)(t)| \to 0 \ \ \mathrm{a.e.} \ t \in (0,T).
\end{equation}
Combining \eqref{eq:ell m to 0} and \eqref{eq:ell p to 0} we see that $\ell(m^n,p^n) \to \ell(m,p)$ a.e.
We deduce \eqref{eq:Q converges} from the definition \eqref{MFG-quantity-def}. 
Therefore $(u^{n}, m^n)$ converges to some $(u,m)$ which is a solution to \eqref{MFG-limit-1} with initial data $m_0$. 

\emph{Step 3 (uniqueness):} 
Let $(u_i,m_i), i=1,2$ be two solutions of \eqref{MFG-limit-1}.
We set
$$
G_i:=q_{u_{i},m_i} \ \mbox{ and } \ \bar G_i := \int_0^L q_{u_{i},m_i}(t,y)\dd m_i(t).
$$
From \eqref{MFG-quantity-def-initial}, we know that
\begin{equation}\label{relation-G-bar G}
G_i = \frac{1}{2}\left(1 - \k \bar G_i - u_{i,x}\right)^+.
\end{equation}

Let $u = u_1 - u_2, m = m_1 - m_2, G = G_1 - G_2, \bar G = \bar G_1 - \bar G_2$.
Using $(t,x)\to e^{-rt}u(t,x)$ as a test function in the equations satisfied by $m_1, m_2$, with some algebra yields
\begin{multline} \label{eq:uniqueness1}
0
= \int_0^T e^{-rt} \int_0^L (G^2_2 - G^2_1 - G_1 u_{x})m_1 + (G^2_1 - G^2_2 + G_2 u_{x})m_2 \dd x \dd t
\\
= \int_0^T e^{-rt} \int_0^L (G_1-G_2)^2(m_1 + m_2)\dd x \dd t
+ \int_0^T e^{-rt} \int_0^L (2G+u_x)(G_2m_2 - G_1m_1) \dd x \dd t.
\end{multline}
Now since $G_2 = 0$ on the set where $1 - \k \bar G_2(t) - u_{2,x} < 0$, we can write
\begin{align*}
(2G+u_x)G_2 
&= \left(\left(1-\k \bar G_1 - u_{1,x}\right)^+ - \left(1 - \k \bar G_2(t) - u_{2,x}\right) + u_{1,x} - u_{2,x}\right)G_2
\\
&= \left(-\k \bar G + \left(1-\k \bar G_1 - u_{1,x}\right)^-\right)G_2.
\end{align*}
Similarly we can write
\begin{align*}
(2G+u_x)G_1 
&= \left(\left(1-\k \bar G_1 - u_{1,x}\right) - \left(1 - \k \bar G_2(t) - u_{2,x}\right)^+ + u_{1,x} - u_{2,x}\right)G_1
\\
&= \left(-\k \bar G - \left(1-\k \bar G_2 - u_{2,x}\right)^-\right)G_1.
\end{align*}
Thus we compute
\begin{multline*}
\int_0^L (2G+u_x)(G_2m_2 - G_1m_1) \dd x \dd t
= \k \bar G^2 + \int_0^L \left(1-\k \bar G_1 - u_{1,x}\right)^- G_2m_2 \dd x \dd t
\\
+  \int_0^L \left(1-\k \bar G_2 - u_{2,x}\right)^- G_1m_1 \dd x \dd t
\geq \k \bar G^2.
\end{multline*}
So from \eqref{eq:uniqueness1} we conclude
\begin{equation}
\label{eq:uniqueness2}
\int_0^T e^{-rt} \int_0^L (G_1-G_2)^2(m_1 + m_2)\dd x \dd t
+ \k \int_0^T e^{-rt} (\bar G_1 - \bar G_2)^2 \dd t = 0.
\end{equation}
In particular, $\bar G_1 \equiv \bar G_2$.
We can then appeal to uniqueness for the Hamilton-Jacobi equation to get $u_1 \equiv u_2$ (cf. \cite[Chapter V]{Parabolic67}).
By \eqref{relation-G-bar G}, this entails that $G_1 \equiv G_2$, and so $m_1 \equiv m_2$ by uniqueness for the Fokker-Planck equation.
\end{proof}




\section{Application of the MFG approach} \label{section::application}

This section is devoted to the proof of Theorem \ref{Theorem-epsilon::Nash-Cournot}.
Namely, we show that the optimal feedback strategy, computed from the MFG system \eqref{MFG-limit-1}, provides an $\ee$-Nash equilibria for the $N $-Player Cournot game, where the error $\ee$ is arbitrarily small as $N\to \infty$. 
Throughout this section $(u,m)$ is the solution to \eqref{MFG-limit-1} starting from some probability measure  $m_0$ satisfying \eqref{Aa1}, and the function $q_{u,m}$ is given by \eqref{MFG-quantity-def-initial}.
Moreover, we define
\begin{equation}
\left\{
\begin{aligned}
&\dd \hat X_{t}^{i} =  - q_{u,m}(t, \hat X_{t}^{i})\dd t + \sigma  \dd W_{t}^{i}-\dd \xi_t^{\hat X^{i}}\\
&X_{0}^{i} = V_{i}, 
\\\end{aligned}
\right.
\end{equation}
and set $\hat q_t^i:= q_{u,m}(t, \hat X_{t}^{i})$.
We recall that the objective functional is defined as
$$
\J_c^{i,N}(q^1,...,q^N):=\E\left\{ \int_{0}^{T} e^{-rs}  \left( 1-\k \bar q_s^i -q_s^i \right) q_s^i \mathds{1}_{s < \tau^{i}} \dd s + e^{-rT}u_{T}( X_{\tau^{i}}^{ i} )  \right\}.
$$
Our goal is to prove that 
$$\J_c^{i,N}\left(q^i; (\hat q^j)_{j\neq i} \right) \leq \ee+\J_c^{i,N}\left(\hat q^1,...,\hat q^N \right) \quad \forall q^i \in \A_c, \ \forall i =1,\ldots,N$$
as long as $N$ is large enough.

Let us set
$$
\hat \tau^{i}:= \inf\left\{ t\geq 0 \ : \ \hat X_t^{i} \leq 0 \right\} \wedge T,
$$
and define the following process:
\begin{equation}\label{empirical process:definition}
\hat \nu_t^N:= \frac{1}{N} \sum_{k=1}^N  \d_{\hat X_{t}^{k}} \1_{t< \hat \tau^{k }}, \quad \forall t \in [0,T],
\end{equation}
where $\d_x$ denotes the Dirac delta measure of the point $x\in \R$. Observe that the above definition makes sense because the stochastic dynamics $( \hat X^{1},..., \hat X^{N})$ exists in the strong sense owing to Lemma \ref{holder-estimates}. In particular, the strategy profile $\left(\hat q^{1},...,\hat q^{N}\right)$ defined in Theorem \ref{Theorem-epsilon::Nash-Cournot} belongs to $\prod_{i=1}^N \mathbb A_c$.
Moreover, by using the probabilistic characterization \eqref{probabilistic-charact-3}, note that for any measurable and bounded function $\phi$ on $\bar Q$ we have 
\begin{equation}\label{propagation of chaos: mean}
\E \left[ \int_0^L \phi \dd \hat \nu_t^N \right] = \int_0^L \phi \dd m(t), \quad  \mbox{ for a.e. } t \in (0,T).
\end{equation}
The above identity is not strong enough to show Theorem \ref{Theorem-epsilon::Nash-Cournot} and we need a stronger condition (cf.~\eqref{requirement-convergence-THm}). Therefore, we need to work harder in order to get  more information on the asymptotic behavior of the empirical process \eqref{empirical process:definition} when $N\to \infty$.

We aim to prove that the \emph{empirical process} $\left(\hat \nu^N \right)_{N\geq 1}$ converges in law to the \emph{deterministic} measure $m$ in a suitable function space,
by using arguments borrowed from \cite{ledger2016-1, ledger2016-2}. For this, we start by showing the existence of sub-sequences $(\hat \nu^{N'} )$ that converges in law to some limiting process $\nu^\ast$. Then, we show that $\nu^\ast$ belongs to $\tilde \P(\bar Q)$ and satisfies the same equation as $m$. Finally, we invoke the uniqueness of weak solutions to the Fokker-Planck equation to deduce full weak convergence toward $m$. 

The crucial step consists in showing that the sequence of the laws of $\left(\hat \nu^N \right)_{N\geq 1}$ is relatively compact on a suitable topological space. This is where the machinery of \cite{ledger2016-2} is convenient. In order to use the analytical tools of that paper, we view the empirical process  as a random variable on the space of \emph{c\`adl\`ag} (right continuous and has left-hand limits) functions, mapping $[0,T]$ into the space of tempered distributions. This function space is denoted $D_{\S'_\R}$ and is endowed with the so called Skorokhod's $\rm M 1$ topology. Note that there are no measurability issues owing to \cite[Proposition 2.7]{ledger2016-2}. Moreover, by virtue of \cite{mitoma1983}, the process $\left(\hat \nu_t^N \right)_{t\in[0,T]}$ has a version that is \emph{c\`adl\`ag} in the strong topology of $\S'_\R$ for every $N\geq 1$,  since $\hat \nu_t^N(\phi) := \int_\R \phi \dd \hat \nu_t^N$ is a real-valued \emph{c\`adl\`ag} process, for every $\phi \in \S_\R$ and $N\geq 1$.
We refer the reader to \cite{ledger2016-2} for the construction of $(D_{\S'_\R}, \rm M1)$, and to \cite{Whitt2002} for general background on Skorokhod's topologies. We shall denote by $(D_\R, \rm M 1)$ the space of $\R$-valued \emph{c\`adl\`ag} functions mapping $[0,T]$ to $\R$, endowed with Skorokhod's $\rm M 1$ topology. 

 The main strengths of working with the $\rm M 1$ topology in our context, are based on the following facts:
\begin{itemize}
\item tightness on $(D_{\S'_\R}, \rm M 1)$ implies the relative compactness on $(D_{\S'_\R}, \rm M 1)$ thanks to \cite[Theorem 3.2]{ledger2016-2});
\item the proof of tightness on $(D_{\S'_\R}, \rm M 1)$ is reduced through the canonical projection to the study of tightness in $(D_\R, \rm M 1)$, for which we have suitable characterizations \cite{Whitt2002, ledger2016-2};
\item bounded monotone real-valued processes are automatically tight on $(D_\R, \rm M 1)$; this is an important feature, that enables  to prove tightness of the  sequence of empirical process laws, by using a suitable decomposition.
\end{itemize}
It is also important to note that this approach could be generalized to deal with the case of a systemic noise, by using a martingale approach as in \cite[Lemma 5.9]{ledger2016-1}. We do not deal with that case in this paper.

More generally, one can replace $\S'_\R$ by any dual space of a \emph{countably Hilbertian nuclear space} (cf. \cite{ledger2016-2} and references therein). Although the class $\S'_\R$ seems to be excessively large for our purposes, we recover measure-valued processes by means of Riesz representation theorem (cf. \cite[Proposition 5.3]{ledger2016-1} for an example in the same context). 

Throughout this part, \emph{we shall use the symbol $\Rightarrow$ to denote convergence in law}. The key technical lemma of this section is the following:

\begin{lemma}\label{propagation-of-chaos-lemma}
As $N \to \infty$, we have $\hat \nu^N \Rightarrow m$ on $(D_{\S'_\R}, \rm M1)$, i.e. for every continuous bounded real-valued function $\Psi$ on $(D_{\S'_\R}, \rm M1)$, it holds that 
$$
\lim_N \E\left[\Psi\left(\hat \nu^N \right) \right] = \Psi(m).
$$
\end{lemma}

The bulk of this section is devoted to the proof of Lemma \ref{propagation-of-chaos-lemma}.
The proof of Theorem \ref{Theorem-epsilon::Nash-Cournot} is completed in Section \ref{sec:proof of epsilon eq}.

\subsection{Proof of Lemma \ref{propagation-of-chaos-lemma}}
\subsubsection{Tightness} At first, we aim to prove the tightness of $(\hat\nu^N)_{N\geq 1}$ on the space $(D_{\S'_\R}, \rm M1)$;
that is, for every $\phi \in \S_\R$ and for all $\ee>0$, there exists a compact subset $K$ of $(D_\R, \rm M 1)$ such that:
$$
\Pp\left(\hat\nu^N(\phi) \in K \right) > 1-\ee \quad \mbox{ for all } \ N\geq 1.
$$
For that purpose, we shall use a convenient characterization of tightness in $(D_\R, \rm M 1)$ (cf. \cite[Theorem 12.12.3]{Whitt2002}).

We start by controlling the concentration of mass at the origin:

\begin{lemma} \label{mass-concentration-origin}
For every $t\in[0,T]$, we have
$$
\sup_{N\geq 1}\E\hat \nu_t^N(0,\ee) \to 0, \quad \mbox{ as } \ee \to 0.
$$
\end{lemma}
\begin{proof}
Let us fix $\ee>0$. Note that, for every $t\in [0,T]$
$$
\E\hat \nu_t^N(0,\ee) = \frac{1}{N}\sum_{i=1}^N \Pp\left( \hat X_t^{i}\in (0,\ee) ;  t< \hat \tau^i \right).
$$
Thus, on the one hand
$$
\sup_{N\geq 1}\E\hat \nu_0^N(0,\ee) = \int_0^\ee \dd m_0  \to 0, \quad \mbox{ as } \ee \to 0.
$$
On the other hand, we have for any $t\in(0,T]$
\begin{equation}\label{convenient-bound-lemma34}
\sup_{N\geq 1}\E\hat \nu_t^N(0,\ee) \leq \sup_{N\geq 1}N^{-1}\sum_{i=1}^N \E\left[ (1-\phi_\ee(\hat X_t^{i})) \1_{ t< \hat \tau^i} \right]
\end{equation}
where $\phi_\ee$ is the cut-off function defined in \eqref{cutt-off-function:specifications}. Thus, by virtue of \eqref{eq:change of probability estimate-bis} we obtain
$$
\sup_{N\geq 1}\E\hat \nu_t^N(0,\ee) \leq C(L,t,\|q_{u,m}\|_\infty) \ee^{1/4},
$$
which entails the desired result.
\end{proof}

The second ingredient is the control of the mass loss increment:
\begin{lemma} \label{mass-loss-increment}
For every $t\in[0,T]$ and $\la >0$
$$
\lim_{h \to 0} \lim\sup_{N} \Pp\left( \left| \eta\left(\hat \nu_t^{N}\right)-\eta\left(\hat \nu_{t+h}^{N}\right) \right| \geq \la\right) =0,
$$
where the map $\mu\to \eta(\mu)$ is defined in \eqref{coefficients-defintion:a,c,eta}.
\end{lemma}
\begin{proof}
The proof is inspired by \cite[Proposition 4.7]{ledger2016-1}.  Let $\ee,h>0$ and $t\in[0,T]$, we have
\begin{multline}\label{first-estimation-lemma45}
\Pp\left(\eta\left(\hat \nu_t^{N}\right)-\eta\left(\hat \nu_{t+h}^{N}\right)\geq \la \right) \\
\leq \Pp\left(\hat \nu_t^N(0,\ee) \geq \la/2\right) + \Pp\left(\eta\left(\hat \nu_t^{N}\right)-\eta\left(\hat \nu_{t+h}^{N}\right) \geq \la ; \hat \nu_t^N(0,\ee) < \la/2\right).
\end{multline}
The reason why we use the latter decomposition will be clear in \eqref{red-estimation-lemma45}.
Owing to Markov's inequality and Lemma \ref{mass-concentration-origin}, one has
$$
\lim\sup_{N}\Pp(\hat \nu_t^N(0,\ee) \geq \la/2) \leq 2 \la^{-1} \sup_{N}\E\hat \nu_t^N(0,\ee) \to 0, \quad \mbox{ as } \ee \to 0.
$$
Now we deal with the second part in estimate \eqref{first-estimation-lemma45}. Define $\mathcal I_t$ to be the following random set of indices:
$$
\mathcal I_t:=\left\{ 1\leq i \leq N \ : \ \hat X_t^{i } \geq \ee  \right\};
$$
then, we have
\begin{multline*}
\Pp\left(\eta\left(\hat \nu_t^{N}\right)-\eta\left(\hat \nu_{t+h}^{N}\right) \geq \la ; \hat \nu_t^N(0,\ee) < \la/2\right) \\
 \leq \sum_{\# \mathcal I \geq N(1-\la/2)} \Pp\left(\eta\left(\hat \nu_t^{N}\right)-\eta\left(\hat \nu_{t+h}^{N}\right) \geq \la \ | \ \mathcal I_t= \mathcal I) \Pp (\mathcal I_t = \mathcal I\right),
\end{multline*}
where $\# \mathcal I$ denotes the number of elements of $\mathcal I \subseteq \{1,2,\ldots,N\}$. Thus, we reduce the problem to the estimation of the dynamics increments; using the same steps as for \eqref{estimate::choice:of:ee} we have 
\begin{multline}\label{red-estimation-lemma45}
\Pp\left(\eta\left(\hat \nu_t^{N}\right)-\eta\left(\hat \nu_{t+h}^{N}\right) \geq \la \ | \ \mathcal I_t= \mathcal I\right) \\
 \leq \Pp\left( \# \left\{i\in \mathcal I \ : \ \inf_{s\in[t,t+h]} \hat X_s^{i}- \hat X_t^{i} \leq - \ee   \right\} \geq \la N/2 \ \Big| \  \mathcal I_t= \mathcal I  \right) \\
 \leq \Pp\left( \# \left\{i\in \mathcal I \ : \ \sup_{s\in[0,h]} B_s^i - \inf_{s\in[0,h]} B_s^i \geq \frac{\ee-h}{\sigma}   \right\} \geq \la N/2  \right),
\end{multline}
where we have used the uniform bound on $q_{u,m}$ of Lemma \ref{holder-estimates}, and where $(B^i)_{1\leq i \leq N}$ is a family of independent Wiener processes. By symmetry, this final probability depends only on $\# \mathcal I$, so that the right hand side above is maximized when $\mathcal I=\{1,...,N\}$. We infer that
$$
\Pp\left(\eta\left(\hat \nu_t^{N}\right)-\eta\left(\hat \nu_{t+h}^{N}\right)\geq \la ; \hat \nu_t^N(0,\ee) < \la/2 \right) 
\leq \Pp\left( \frac{1}{N}\sum_{i=1}^N \1_{ \left\{ \sup_{s\in[0,h]} B_s^i - \inf_{s\in[0,h]} B_s^i \geq \frac{\ee-h}{\sigma} \right\}} \geq \la/2  \right).
$$
In the same way as for \eqref{estimate::choice:of:ee}, we choose $\ee(h)=h^{1/2} \log(1/h)$ so that $\lim_{h\to 0^+}\ee(h)=0$, and use Markov's inequality to get
$$
\Pp\left(\eta\left(\hat \nu_t^{N}\right)-\eta\left(\hat \nu_{t+h}^{N}\right)\geq \la ; \hat \nu_t^N(0,\ee) < \la/2 \right) 
\leq \frac{4\sigma}{\la(\log(1/h)-h^{1/2})}.
$$
This entails the desired result by taking the limit $h \to 0^+$.

Now we deal with the case of a left hand limit. Let $t\in (0,T]$ and $h \mapsto \ee(h)$ as defined above. Using a similar decomposition as before, we have for small enough $h>0$
\begin{multline*}
\Pp\left(\eta\left(\hat \nu_{t-h}^{N}\right)-\eta\left(\hat \nu_{t}^{N}\right)\geq \la \right) \\
\leq \Pp\left(\hat \nu_{t-h}^N(0,\ee) \geq \la/2\right) + \Pp\left(\eta\left(\hat \nu_{t-h}^{N}\right)-\eta\left(\hat \nu_{t}^{N}\right) \geq \la ; \hat \nu_{t-h}^N(0,\ee) < \la/2\right).
\end{multline*}
Appealing to Markov's inequality, estimate \eqref{convenient-bound-lemma34}, and estimate \eqref{eq:change of probability estimate-bis} of Section \ref{section2}, we have for small enough $h>0$
$$
\Pp\left(\hat \nu_{t-h}^N(0,\ee) \geq \la/2 \right) \leq 2 \la^{-1} \E\hat \nu_{t-h}^N(0,\ee) \leq  2C \la^{-1}  \left(1-e^{-\pi^2 t/8L^2} \right)^{-1/4}\ee^{1/4},
$$
whence
$$
\lim_{h\to 0^+}\lim\sup_N \Pp \left(\hat \nu_{t-h}^N(0,\ee(h)) \geq \la/2 \right) =0.
$$
On the other hand, we show by using the same steps as in \eqref{red-estimation-lemma45} that
$$
\Pp\left(\eta\left(\hat \nu_{t-h}^{N}\right)-\eta\left(\hat \nu_{t}^{N}\right) \geq \la ; \hat \nu_{t-h}^N(0,\ee) < \la/2\right) \leq \frac{4\sigma}{\la(\log(1/h)-h^{1/2})}.
$$
This entails the desired result by taking the limit $h\to 0^+$.
\end{proof} 

We are now in position to show tightness on $(D_{\S'_\R}, {\rm M1})$. 

\begin{proposition}[Tightness]\label{proposition-tightness}
The sequence of the laws of $(\hat \nu^N)_{N\geq 1}$ is tight on the space $(D_{\S'_\R}, \rm M1)$.
\end{proposition}
\begin{proof}
We present a brief sketch to explain the main arguments, and refer to \cite[Proposition 5.1]{ledger2016-1} for a similar proof.

Thanks to \cite[Theorem 3.2]{ledger2016-2}, it is enough to show that the sequence of the laws of $\left(  \hat \nu^N(\phi) \right)_{N\geq 1}$ is tight on  $(D_\R, \rm M1)$ for any $\phi \in \S_\R$. To prove this, one can use the conditions of \cite[Theorem 12.12.3]{Whitt2002}, which can be rewritten in a convenient form by virtue of \cite{Avram1989}. From \cite[Proposition 4.1]{ledger2016-2} , we are done if we achieve the two following steps:
\begin{enumerate}
\item \label{req1:prop36} find $\a,\b,c>0$, such that
$$
\Pp\left( H_\R\left(\hat \nu_{t_1}^N(\phi), \hat \nu_{t_2}^N(\phi), \hat \nu_{t_3}^N(\phi)   \right) \geq \la   \right) \leq c \la^{-\a}|t_3-t_1|^{1+\b},
$$
for any $N\geq 1$, $\la>0$ and $0\leq t_1<t_2<t_3\leq T$, where
$$
H_\R\left(x_1,x_2,x_3  \right):= \inf_{0\leq \g \leq 1} | x_2-(1-\g)x_1- \g x_3 | \quad \mbox{ for } x_1,x_2,x_3 \in \R;
$$
\item show that
$$
\lim_{h \to 0^+} \lim_{N}\Pp\left( \sup_{t\in(0,h)}| \hat \nu_{t}^N(\phi)- \hat \nu_{0}^N(\phi)| + \sup_{t\in(T-h,T)}| \hat \nu_{T}^N(\phi)- \hat \nu_{t}^N(\phi)| \geq \la  \right)=0.
$$
\end{enumerate}
The key step is to consider the following decomposition \cite[Proposition 4.2]{ledger2016-2}:
\begin{equation}\label{decomposition:trick}
\bar \nu_t^N(\phi):=\frac{1}{N} \sum_{k=1}^N \phi( \hat X_{t\wedge \hat \tau^k }^{k}) = \hat \nu_t^N(\phi) + \phi(0)\mathcal E_t^N,
\end{equation}
where 
$$
\mathcal E_t^N:= 1-\eta\left(\hat \nu_t^{N}\right)
$$
is the exit rate process, which quantifies the fraction of firms out of market.
Since $\left(\mathcal E_t^N \right)_{t\in[0,T]}$ is monotone increasing we have
$$
\inf_{0 \leq \g \leq 1} \left| \mathcal E_{t_2}^N-(1-\g)\mathcal E_{t_1}^N-\g \mathcal E_{t_3}^N \right|=0,
$$
so that
$$
H_\R\left(\hat \nu_{t_1}^N(\phi), \hat \nu_{t_2}^N(\phi), \hat \nu_{t_3}^N(\phi)   \right) \leq  \left| \bar \nu_{t_1}^N(\phi)- \bar \nu_{t_2}^N(\phi)  \right| + \left| \bar \nu_{t_2}^N(\phi) -\bar \nu_{t_3}^N(\phi)   \right|.
$$
Thus, by virtue of Markov's inequality 
\begin{multline*}
\Pp\left( H_\R\left(\hat \nu_{t_1}^N(\phi), \hat \nu_{t_2}^N(\phi), \hat \nu_{t_3}^N(\phi)   \right) \geq \la   \right)\\ 
\leq 8\la^{-4} \left( \E\left| \bar \nu_{t_1}^N(\phi)- \bar \nu_{t_2}^N(\phi)  \right|^4 +   \E\left| \bar \nu_{t_2}^N(\phi)- \bar \nu_{t_3}^N(\phi)  \right|^4 \right).
\end{multline*}
Therefore, we deduce requirement \eqref{req1:prop36} from the following estimate:
\begin{multline}\label{estimate1:tightness}
\forall s,t \in[0,T], \\
\E\left| \bar \nu_{t}^N(\phi)- \bar \nu_{s}^N(\phi)  \right|^4 \leq \| \phi_x\|_{\infty}^4 \frac{1}{N}\sum_{k=1}^{N}  \E| \hat X_{t\wedge \hat \tau^{k}}^{k} -\hat X_{s\wedge \hat \tau^{k}}^{k}   |^4 
\leq C \| \phi_x\|_{\infty}^4 |t-s|^2;
\end{multline}
where we have used H\"older's inequality and the Burkholder-Davis-Gundy inequality \cite[Thm IV.42.1]{Rogers2000martingales}. 

The second requirement is also obtained by using the latter estimate, decomposition \eqref{decomposition:trick},  and Lemma \ref{mass-loss-increment}. In fact, we have
\begin{multline*}
\Pp\left( \sup_{t\in(0,h)}| \hat \nu_{t}^N(\phi)- \hat \nu_{0}^N(\phi)|  \geq \la  \right)  \\
\leq  \Pp\left( \sup_{t\in(0,h)}|  \bar \nu_{t}^N(\phi)- \bar \nu_{0}^N(\phi)|  \geq \la/2  \right) + \Pp\left(|\phi(0)| \mathcal E_h^N \geq \la/2 \right),
\end{multline*}
so that the desired result follows thanks to \eqref{estimate1:tightness}, and Lemma \ref{mass-loss-increment}. By the same way, we deal with the second term $\Pp\left( \sup_{t\in(T-h,T)}| \hat \nu_{T}^N(\phi)- \hat \nu_{t}^N(\phi)|  \geq \la  \right)$. 
\end{proof}

\subsubsection{Full convergence}
We arrive now at the final ingredient for the proof of Lemma \ref{propagation-of-chaos-lemma}. Let us set 
$$
\C^{test} := \left\{ \phi \in \C_c^\infty([0,T) \times \bar Q) \ \big| \ \phi(t,0) = \phi_x(t,L) = 0, \ \ \forall t \in (0,T)  \right\}.
$$

We start by deriving an equation for $(\hat \nu_t^N)_{t\in[0,T]}$.

\begin{proposition}\label{propositon-finite-evolution-equation}
For every $N \geq 1$ and $\phi \in \C^{test}$, it holds that
$$
\int_0^L  \phi(0,.)\dd \hat \nu_0^N = \int_0^T \int_0^L \left( -\phi_t-\frac{\sigma^2}{2}\phi_{xx} + q_{u,m} \phi_x  \right) \dd \hat \nu_s^N \dd s + I_N(\phi)\quad a.s.,
$$
where 
$$
I_N(\phi) :=-\frac{\sigma}{N} \sum_{k=1}^N \int_0^T \phi_x\left(s,\hat X_s^{k}\right) \1_{s<\hat \tau^{k}} \dd W_s^k.
$$
\end{proposition}
\begin{proof}
Let us consider $\phi \in \C^{test}$. First observe that for any $k\in \{ 1,...,N\}$, and $t\in [0,T]$
$$
\hat X_{t\wedge \hat \tau^{k}}^{k} = V_k - \int_0^t \hat q_s^k \1_{s < \tau^{\hat q^k}} \dd s + \sigma W_{t\wedge \tau^{\hat q^k}}^k - \xi_t^{\hat X^{k}}.
$$
Hence, for any $k\in \{ 1,...,N\}$, the random process $\left(\hat X_{t\wedge \hat \tau^{k}}^{k} \right)_{t\in [0,T]}$ is a continuous semimartingale, and by applying It\^o's rule we have:
\begin{multline*}
\phi (T, \hat X_{\hat \tau^{k}}^{k} ) - \phi(0,V_k) + \int_0^T  \phi_{x} \left(s, \hat X_{s\wedge \hat \tau^{k}}^{k} \right)  \dd \xi_s^{\hat X^k}\\
=\int_0^T  \left\{ \frac{\sigma^2}{2}\phi_{xx} (s,\hat X_{s}^{k} ) - q_{u,m}(s, \hat X_{s}^{k})\phi_{x} (s,\hat X_{s}^{k})  \right\}  \1_{s< \hat \tau^{k}} \dd s\\
+ \int_0^T \phi_t \left(s,\hat X_{s\wedge \hat \tau^{k}}^{k} \right) \dd s+ \sigma \int_0^T  \phi_{x} \left(s,\hat X_{s}^{k} \right)  \1_{s< \hat \tau^{k}} \dd W_s^k.
\end{multline*}
By using the boundary conditions satisfied by $\phi$, and noting that $\phi_t(t,0)=0$ for any $t\in(0,T)$, we deduce that
\begin{multline*}
- \phi(0,V_k)- \sigma \int_0^T  \phi_{x} \left(s,\hat X_{s}^{k} \right)  \1_{s<\hat \tau^{k}} \dd W_s^k \\
=\int_0^T  \left\{ \phi_t \left(s,\hat X_{s}^{k} \right) +\frac{\sigma^2}{2}\phi_{xx} \left(s,\hat X_{s}^{k} \right)  - q_{u,m}(s, \hat X_{s}^{k})\phi_{x} \left(s,\hat X_{s}^{k} \right)  \right\} \1_{s< \hat \tau^{k}} \dd s
\end{multline*}
The desired result follows by summing over $k\in\{1,...,N\}$, and multiplying by $N^{-1}$.
\end{proof}

By virtue of \cite[Theorem 3.2]{ledger2016-2}, the tightness of the sequence of laws of $(\hat \nu^N)_{N\geq 1}$ ensures that this sequence is relatively compact on $(D_{\S'_\R}, \rm M1)$. Consequently, Proposition \ref{proposition-tightness} entails the existence of a subsequence (still denoted $(\hat \nu^N)_{N\geq 1}$) such that 
$$
\hat \nu^N \Rightarrow \hat \nu^\ast, \quad \mbox{ on } (D_{\S'_\R}, \rm M1).
$$
Thanks to \cite[Proposition 2.7 (i)]{ledger2016-2},
$$
\forall \phi\in \S_\R, \quad  \hat \nu^N(\phi) \Rightarrow  \hat \nu^\ast(\phi), \quad \mbox{ as } N\to \infty, \quad \mbox{ on } (D_\R, \rm M 1).
$$

To avoid possible confusion about multiple distinct limit points, we will denote $\hat \nu^\ast$ any limiting processes that realizes one of these limiting laws. First, we note that $\hat \nu^{\ast}$ is a $\tilde \P(\bar Q)$-valued process:

\begin{proposition}
For every $t\in[0,T]$, $\hat \nu_t^\ast$ is almost surely supported on $\bar Q$ and belongs to $\tilde \P(\bar Q)$. 
\end{proposition}
\begin{proof}
This follows from the ``Portmanteau theorem" and the Riesz representation theorem. We omit the details and refer to \cite[Proposition 5.3]{ledger2016-1}.
\end{proof}

Next, we recover the partial differential equation satisfied by the process $(\hat \nu_t^\ast)_{t\in[0,T]}$. 

\begin{lemma}\label{equation-accumulation-point}
For every $\phi \in \C^{test}$, it holds that
$$
\int_0^L  \phi(0,.) \dd m_{0}  +\int_0^T \int_0^L \left( \phi_t+\frac{\sigma^2}{2}\phi_{xx} - q_{u,m} \phi_x  \right) \dd \hat \nu_s^{\ast} \dd s = 0 \quad { a.s.}
$$
\end{lemma}
\begin{proof}
Let us consider $\phi \in \C^{test}$ and set:
$$
\mu(\phi):= \int_0^L  \phi(0,.)  \dd m_0 +\int_0^T \int_0^L \left( \phi_t+\frac{\sigma^2}{2}\phi_{xx} - q_{u,m} \phi_x  \right) \dd \hat \nu_s^{\ast} \dd s;
$$
and
$$
\mu_N(\phi) := \int_0^L \phi(0,.) \dd m_0 +\int_0^T \int_0^L \left( \phi_t+\frac{\sigma^2}{2}\phi_{xx} - q_{u,m}  \phi_x  \right) \dd \hat \nu_s^{N} \dd s.
$$
Owing to Proposition \ref{propositon-finite-evolution-equation} we have
$$
\mu_N(\phi)= I_N(\phi)+ \int_0^L \phi(0,.)\dd(m_0- \hat \nu_0^{N}).
$$
Note that
$$
\E I_N(\phi)^2 \leq C \| \phi_x \|_{\infty}^2 N^{-1}.
$$
Hence, by appealing to Horowitz-Karandikar inequality (see e.g. \cite[Theorem 10.2.1]{product:measures}) we deduce that
$$
\E\mu_N^2(\phi) \leq C \| \phi_x \|_{\infty}^2 N^{-2/5}.
$$
Consequently, to conclude the proof it is enough to show that
$$
\mu_N(\phi) \Rightarrow  \mu(\phi) \quad \mbox{ as } N\to \infty.
$$

Let $\A$ be the set of elements in $D_{\S'_\R}$ that take values in $\tilde \P(\bar Q)$, and consider a sequence $(\psi^N) \subset \A$ which converges to some $\psi$ in $\A$ with respect to the $\rm M1$ topology.  
Let ${\bf q}_{u,m}$  be a continuous function on $[0,T]\times \R$, which satisfies the following conditions: 
\begin{subequations}
\begin{equation}\label{extension-definition}
{{\bf q}_{u,m}}_{|_{\overline{Q}_T}} \equiv q_{u,m}; \quad \| {\bf q}_{u,m} \|_{\infty} = \| q_{u,m}\|_{\infty}; \quad  \forall t\in[0,T],  \ \ {\rm supp } \ {\bf q}_{u,m}(t,.)\subset (-L, 2L).
\end{equation}
We also define the sequence 
\begin{equation}\label{extension-regularization}
{\bf q}_{u,m}^n(t,x) := \left( {\bf q}_{u,m}(t,.)\ast \xi_n  \right)(x), \quad n\geq 1,
\end{equation}
\end{subequations}
where $\xi_n(x):= n\xi(nx)$ is a compactly supported mollifier on $\R$.

We have
\begin{multline*}
J:= \left| \int_0^T \int_0^L q_{u,m} \phi_x   \dd \psi_s^N \dd s- \int_0^T\int_0^L q_{u,m}  \phi_x   \dd  \psi_s \dd s \right| \\
= \left| \int_0^T \int_\R {\bf q}_{u,m} \phi_x   \dd \psi_s^N \dd s- \int_0^T\int_\R {\bf q}_{u,m}  \phi_x   \dd  \psi_s \dd s \right| \\
\leq 2\| \phi_x\|_{\infty} \left\| {\bf q}_{u,m}^n-{\bf q}_{u,m} \right\|_{\infty}  \\
+ \left| \int_0^T  \int_\R {\bf q}_{u,m}^n \phi_x  \dd( \psi_s^{N} -   \psi_s)  \dd s  \right|  =: J_1+J_2. 
\end{multline*}
Since ${\bf q}_{u,m}^n(s,.) \phi_x(s,.) \in \S_\R$ for any $s\in[0,T]$, then $J_2$ vanishes as $\psi^N \to \psi$ . On the other hand, note that $J_1$ also vanishes as $n \to +\infty$ so that we obtain $\lim_N J=0$. 
Moreover, one easily checks that
$$
\int_0^T \int_0^L  F  \dd \psi_s^N \dd s \to \int_0^T\int_0^L F  \dd  \psi_s \dd s, \quad F\equiv \phi_t, \phi_{xx} \ \ \mbox{ as } \ \ N\to +\infty.
$$
Therefore, by virtue of the continuous mapping theorem, we obtain that $\mu_N(\phi) \Rightarrow  \mu(\phi)$, which concludes the proof.
\end{proof}

We are now in position to prove Lemma \ref{propagation-of-chaos-lemma}.

\begin{proof}[Proof of Lemma \ref{propagation-of-chaos-lemma}] From Lemma \ref{equation-accumulation-point}, we know that $\dd {\bf \nu}^\ast=\dd \hat \nu_t^{\ast} \dd t$ and $\dd {\bf m}=\dd m(t) \dd t$ both satisfy (almost surely) the same Fokker-Planck equation in the sense of measures (cf. Appendix \ref{App2}). By invoking the uniqueness of solutions to that equation (cf. Proposition \ref{uniqueness:measure-valued}), we deduce that  $\hat \nu^{\ast} \equiv m$ almost surely. Since all converging sub-sequences converge weakly toward $m$, we infer that 
$
\hat \nu^N \Rightarrow m,$ on $(D_{\S'_\R}, \rm M1).
$
\end{proof}

\subsection{Proof of Theorem \ref{Theorem-epsilon::Nash-Cournot}} \label{sec:proof of epsilon eq}
We start by collecting the following technical result whose proof is given in Appendix \ref{App1}.

\begin{lemma}\label{continuity:M1}
Fix $n\geq 1$, define $\A$ to be all elements in $D_{\S'_\R}$ that take values in $\tilde \P(\bar Q)$, and let $\Psi_m$ (resp. $\Psi_{\bf q}^n$) be the map defined from $D_{\S'_\R} $ into $D_{\S'_\R} $ (resp. from $\A$ into $D_\R$) such that
$$
\Psi_m(\nu)(t):= \nu(t)- m(t)   \quad \mbox{ and } \quad   \Psi_{\bf q}^n(\nu)(t) := \left| \int_{\R} {\bf q}_{u,m} ^{n}(t,.) \dd \nu(t)  \right|.
$$
Then $\Psi_m,  \Psi_{\bf q}^n$  are continuous with respect to the $\rm M1$ topology.
\end{lemma}

Let us now explain the proof of Theorem \ref{Theorem-epsilon::Nash-Cournot}.
We shall proceed by contradiction, assuming  that \eqref{epsilon-Nash-th2.13-Cournot} does not hold. Then there exists $\ee_0>0$, a sequence of integers $N_k$ such that $\lim_k N_k = +\infty$, and sequences $(i_k) \subset \{1,..., N_k\}$, $(q^{i_k}) \subset \mathbb A_c $, such that
\begin{equation}\label{contradiction-assumption-Th31}
 \J_c^{i_k,N_k}\left( q^{i_k} ; (\hat q^j)_{j\neq i_k} \right) > \ee_0 + \J_c^{i_k,N_k}\left(\hat q^1,...,\hat q^N \right), \ \ \forall k \geq 0.
\end{equation}
We derive a contradiction by estimating the difference between $\J_c^{i_k,N_k}$ and the mean field objective $\J_c$, which we recall from Section \ref{sec:main results}:
	\begin{equation} \label{eq:mean field objective}
	\J_c(\rho):=  \E\left\{\int_0^T e^{-r s} \left(1-\k \bar q_s-\rho_s  \right)\rho_s \1_{s<\tau^{\rho}}\dd s + e^{-rT}u_{T}\left(X_{T}^\rho \right)   \right\}, \ \bar q = \int_0^L q_{u,m} \dd m
	\end{equation}
	where
	$$
	\dd X_t^\rho=- \rho_t \1_{t<\tau^\rho} \dd t + \sigma \1_{t<\tau^\rho} \dd W_{t} - \dd \xi_t^{X^\rho}.
	$$
Using Lemma \ref{propagation-of-chaos-lemma}, we will show that this difference goes to zero as $N_k \to +\infty$.

Let us set for any $k\geq 0$,
\begin{equation*}
\left\{
\begin{aligned}
&\dd X_{t}^{i_k} :=  - q_t^{i_k} \dd t + \sigma  \dd W_{t}^{i_k}-\dd \xi_t^{X^{i_k}},\ \  X_{0}^{i_k} = V_{i_k},\\
&  \tau^{i_k}:=\inf\{ t\geq 0 \ : \ X_{t}^{i_k}\leq 0 \}\wedge T,
\end{aligned}
\right.
\end{equation*}
and define 
$$
\Z_{1,T}^k := \int_0^T  q_s^{i_k} \1_{s< \tau^{i_k}} \dd s, \quad \mbox{ and } \quad  \Z_{2,T}^k := \int_0^T  \left| q_s^{i_k} \right|^2 \1_{s< \tau^{i_k}} \dd s.
$$
Recall that all elements of $\mathbb A_c$ are non-negative, so that  $\Z_{1,T}^k \geq 0$ for any $k\geq 0$. We start by collecting estimates on $\left( \Z_{1,T}^{i_k}\right)_{k\geq 0}$ and $\left( \Z_{2,T}^{i_k}\right)_{k\geq 0}$.
Observe that for any $t\in [0,T]$,
$$
X_{t \wedge \tau^{i_k}}^{i_k} = V_k - \int_0^t  q_s^{i_k} \1_{s< \tau^{{i_k}}} \dd s  + \sigma W_{ t\wedge \tau^{i_k}}^{i_k} - \xi_t^{ X^{i_k}}, \quad  \forall k\geq0.
$$
Since the local time is nondecreasing, we infer that
$$
0 \leq \Z_{1,T}^k \leq V_{i_k} - X_{ \tau^{i_k}}^{i_k} + \sigma W_{ \tau^{i_k}}^{i_k}, \quad \forall k\geq 0
$$
holds almost surely. By means of the optional stopping theorem, we deduce that
\begin{equation}\label{th41-estim2}
\sup_{k\geq 0}\E\left[ \Z_{1,T}^{k}\right]  \leq L.
\end{equation}
Moreover, recall that 
$$
 \J_c^{i_k,N_k}\left( q^{i_k} ; (\hat q^j)_{j\neq i_k} \right) =\E\left\{ \int_{0}^{T} e^{-rs}  \left( 1-\k \overline{\hat q_s}^{i_k}-q_s^{i_k} \right) q_s^{i_k} \mathds{1}_{s < \tau^{i_k}} \dd s + e^{-rT}u_{T}( X_{\tau^{i_k}}^{i_k} )  \right\},
$$ 
where for any $k\geq 0$
$$
\overline{\hat q_s}^{i_k} = \frac{1}{N_{k}-1}\sum_{j \neq i_k} q_{u,m}(s, \hat X_{s}^{j} )\1_{s < \hat \tau^{j}}.
$$
Thus, for any $k\geq 0$
$$
e^{-r T}\E \left[ \Z_{2,T}^k \right] \leq  \| u_T \|_{\infty} + \E\left\{ \int_{0}^{T} e^{-rs}  \left| 1-\k \overline{\hat q_s}^{i_k} \right| q_s^{i_k} \mathds{1}_{s <\tau^{i_k}} \dd s  \right\} -  \J_c^{i_k,N_k}\left( q^{i_k}; (\hat q^j)_{j\neq i_k} \right).
$$
By virtue of \eqref{contradiction-assumption-Th31} and the uniform bound on $q_{u,m}$ that is given in \eqref{uniform-bound::G-lemma23}, we deduce that
$$
e^{-r T}\E \left[ \Z_{2,T}^k \right] \leq  2\| u_T \|_{\infty} + (\k+1)\sup_{k\geq 0}\E\left[ \Z_{1,T}^{k}\right] +  C(\k,T),
$$
so that
\begin{equation}\label{th41-estim2-bis}
\sup_{k\geq 0}\E\left[ \Z_{2,T}^{k}\right]  \leq C(T,\k, \| u_T \|_{\infty}, L).
\end{equation}

On the other hand, we have for any $k\geq 0$,
\begin{multline*}
 \J_c^{i_k,N_k}\left( q^{i_k} ; (\hat q^j)_{j\neq i_k} \right) \\
  \leq  \E\left\{ \int_{0}^{T} e^{-rs}  \left( 1-\k \int_0^L q_{u,m}(s,.) \dd \hat \nu_s^{N_k} -q_s^{i_k} \right) q_s^{i_k} \mathds{1}_{s <  \tau^{i_k}} \dd s + e^{-rT}u_{T}( X_{\tau^{i_k}}^{i_k} ) \right\} \\
  + \k \left(\frac{N_k}{N_k-1} -1 \right) + \frac{\k}{N_k} \sup_{k\geq 0}\E\left[ \Z_{1,T}^{i_k}\right].
\end{multline*}
Thus, for any $k\geq 0$
\begin{multline*}
 \J_c^{i_k,N_k}\left( q^{i_k} ; (\hat q^j)_{j\neq i_k} \right)- \J_c( q^{i_k})- CN_k^{-1}\\
\leq \k \E \left[ \int_0^T e^{-r s} q_s^{i_k} \mathds{1}_{s < \tau^{q^{i_k}}} \left| \int_\R  {\bf q}_{u,m}(s,.)  \dd \left(m(s)-\hat \nu_s^{N_k} \right) \right|  \dd s  \right]\\
\leq  \k \E \left[ \int_0^T e^{-r s} q_s^{i_k} \mathds{1}_{s < \tau^{q^{i_k}}}\left| \int_\R  {\bf q}_{u,m}^n (s,.)  \dd \left(m(s)-\hat \nu_s^{N_k} \right) \right|  \dd s  \right]\\
+ \k \E \left[ \int_0^T e^{-r s} q_s^{i_k} \mathds{1}_{s < \tau^{q^{i_k}}} \dd s \right] \left\|  {\bf q}_{u,m}^n  - {\bf q}_{u,m} \right\|_{\infty},
\end{multline*}
where $\J_c$ is given by \eqref{eq:mean field objective} and ${\bf q}_{u,m}, {\bf q}_{u,m}^n $ are given by \eqref{extension-definition}-\eqref{extension-regularization}.

Let us fix $\ee >0$.
Since $\left({\bf q}_{u,m}^n \right)_{n\geq 1}$ converges uniformly toward ${\bf q}_{u,m}$ on $[0,T]\times \R$, we can choose $n$ large enough and independently of $k\geq 0$ so that
\begin{multline}\label{th41-estim3}
\J_c^{i_k,N_k}\left( q^{i_k} ; (\hat q^j)_{j\neq i_k} \right)- \J_c( q^{i_k}) \\
\leq \k \E \left[ \Z_{2,T}^k \right]^{1/2}  \E\left[  \int_0^T  \left| \int_\R  {\bf q}_{u,m}^n(s,.)  \dd \left(\hat \nu_s^{N_k}- m(s) \right)  \right|^2 \dd s  \right]^{1/2}
+ \k  \ee \E \left[\Z_{1,T}^k\right] + C N_k^{-1}.
\end{multline}
Appealing to Lemma \ref{propagation-of-chaos-lemma}, Lemma \ref{continuity:M1} and the continuous mapping theorem we have
\begin{equation}\label{requirement-convergence-THm}
\lim_N \E\left[  \int_0^T  \left| \int_\R  {\bf q}_{u,m}^n(s,.)  \dd \left(\hat \nu_s^{N_k}- m(s) \right)  \right|^2 \dd s  \right]=0.
\end{equation}
Thus, by combining \eqref{th41-estim2}, \eqref{th41-estim2-bis}, and \eqref{th41-estim3}:
$$
\J_c^{i_k,N_k}\left( q^{i_k} ; (\hat q^j)_{j\neq i_k} \right)- \J_c( q^{i_k}) \leq  C(T,\k, \| u_T \|_{\infty}, L) \ee
$$
for big enough $k\geq 0$. Whence, by means of Lemma \ref{Lemma-optimal:control-Cournot}:
\begin{equation*}
\J_c^{i_k,N_k}\left( q^{i_k} ; (\hat q^j)_{j\neq i_k} \right) \leq C(T,\k, \| u_T \|_{\infty}, L)\ee + \J_c(\rho^\ast)
\end{equation*}
for big enough $k\geq 0$. 
In the same manner, one can show that
$$
\J_c(\rho^{\ast}) \leq C \ee + \J_c^{i_k,N_k}\left(\hat q^1,...,\hat q^N \right)
$$
holds for big enough $k\geq 0$. Hence, going back to  \eqref{contradiction-assumption-Th31} and using the above estimates, we obtain
$$
\ee_0  < C(T,\k, \| u_T \|_{\infty}, L)\ee.
$$
We deduce the desired contradiction by choosing $\ee$ suitably small.

\appendix

\section{Proofs of some elementary or technical results} \label{App1}

We start by giving a proof to Lemma \ref{Lemma-optimal:control-Cournot}.

\begin{proof}[Proof of Lemma \ref{Lemma-optimal:control-Cournot}]
This kind of verification results is standard: one checks that the candidate optimal control is indeed the maximum using the equation satisfied by $u$; which is the value function. Let $\rho$ be an admissible control ($\mathbb F$-adapted and satisfying the constraints). 
Since the local time is monotone, then $X^\rho$ is a semimartingale and with the use of It\^o's rule we obtain
\begin{multline*}
\E\left[e^{-rT}u_T\left(X_{\tau^{\rho}}^\rho \right) \right] = \\
\E\left[ u(0,X_0^\rho)+\int_0^{\tau^{\rho}} e^{-r s} \left\{ u_t(s,X_s^\rho)-ru(s,X_s^\rho)-\rho_s u_x(s,X_s^\rho)+\frac{\sigma^2}{2} u_{xx}(s,X_s^\rho)   \right\} \dd s   \right]\\
= \E\left[ u(0,X_0^\rho)-\int_0^{\tau^{\rho}} e^{-r s} \left\{q_{u,m}^{2}(s,X_s^\rho) +\rho_s u_x(s,X_s^\rho)  \right\} \dd s   \right],
\end{multline*}
where we have used the boundary value problem satisfied by $u$ and the fact that $u_t,u_x,u_{xx}$ are continuous on $(0,T)\times (0,L]$ (cf. \eqref{u-regularity}).

By using definition \eqref{MFG-quantity-def-initial}, note that 
\begin{equation*}
q_{u,m}^{2}  = \frac{1}{4}\left| (1-\k \bar q -u_x) \vee 0 \right|^{2}  = \sup_{\rho \geq  0} \rho(1-\k \bar q- \rho-u_x)
= q_{u,m}(1-\k \bar q-q_{u,m}-u_x).
\end{equation*}
Therefore
$$
\E\left[e^{-rT}u_T\left(X_{\tau^{\rho}}^\rho \right) \right] \leq  \E\left[ u(0,X_0^\rho)-\int_0^{\tau^{\rho}} e^{-r s} \rho_s (1-\k \bar q - \rho_s)  \dd s   \right],
$$
so that
$$
\int_0^L u(0,.) \dd m_0=  \E\left[ u(0,X_0^\rho)\right] \geq 
\E\left[ \int_0^{\tau^{\rho}} e^{-r s} (1-\k \bar q-\rho_s) \rho_s   \dd s + e^{-rT}u_T\left(X_{ \tau^{\rho}}^\rho \right) \right].
$$
By virtue of Lemma \ref{holder-estimates}, we know that the process $( X_t^{\rho^\ast} )_{t\in[0,T]}$ exists in the strong sense. Replacing $\rho$ by $\rho^\ast$ in the above computations, inequalities become equalities and we easily infer that
$$
\J_c(\rho^\ast)= \int_0^L u(0,.) \dd m_0.
$$
Thus \eqref{equation-1-lemma211-cournot} is proved. 
\end{proof}

Next, we give a proof to Lemma \ref{continuity:M1}.
\begin{proof}[Proof of Lemma \ref{continuity:M1}]
Throughout the proof, we shall use notations of \cite{ledger2016-2, Whitt2002}.

\emph{Step 1 (continuity in $\S'_\R$):} By virtue of Theorem \ref{theorem:wellposedness}, we know that $t \to m(t)$ is continuous on $[0,T]$ with respect to the strong topology of $\S'_\R$. Let $\phi \in \S'_\R$, we aim to compute the modulus of continuity of $t \to \int_\R \phi \dd m(t)$.
For this, we shall appeal to the probabilistic characterization \eqref{probabilistic-charact-strong}, thanks to Remark \ref{probabilistic-measures-every time}. 
We have for any $h>0$
\begin{multline}\label{lemmaA1-estim1}
\left| \int_\R \phi \dd(m(t+h)-m(t))  \right| \leq  \E\left| \phi(X_{t+h})\1_{t+h<\tau} - \phi(X_t) \1_{t<\tau}  \right| \\
\leq C\| \phi \|_{\C^1} \left(  \Pp(t<\tau) - \Pp(t+h<\tau) + \E\left| X_{t+h} -X_t  \right|   \right).
\end{multline}
Following the same steps as for \eqref{estimate::choice:of:ee-bis}-\eqref{estimate::choice:of:ee}, and using Burkholder-Davis-Gundy inequality, we obtain for small enough $h>0$
$$
\left| \int_\R \phi \dd(m(t+h)-m(t))  \right| 
\leq C\| \phi \|_{\C^1}\omega_m(h),
$$
where 
$$
\omega_m (h) := h^{1/2} +\left(\log(1/h)-h^{1/2} \right)^{-1}+ \sup_{s\in[0,T]}\int_0^{L}(1-\phi_{h^{1/2}\log(1/h)}(x))  m(s,x)  \dd x, 
$$
and $\phi_\e$ is the cut-off function defined in \eqref{cutt-off-function:specifications}. 
In order to get $\lim_{h\to 0^+} \omega_m (h)=0$, we need to prove that
\begin{equation*}
\lim_{h\to 0^+}\sup_{s\in[0,T]}\int_0^{L}(1-\phi_{h^{1/2}\log(1/h)}(x))  m(s,x)  \dd x=0.
\end{equation*}
This ensues easily from Dini's Lemma, by choosing  the sequence
$(\phi_\e)_{\e>0}$ to be monotonically increasing.

\emph{Step 2 (continuity of $\Psi_m$):} Let $\e>0$, $x,y \in D_{\S'_\R}$, $B$ be any bounded subset of $\S_\R$, and $\la_{x}:=(z_x,t_x), \la_y:=(z_y,t_y)$ be a  parametric representations of the graphs of $x$ and $y$ respectively, such that
$$
g_B(\la_x,\la_y):= \sup_{s\in[0,1]} p_B(z_x(s)-z_y(s))\vee \left| t_x(s)-t_y(s)\right| \leq \e,
$$
where $p_B(\nu):=\sup_{x\in B}|\nu(x)|$.
Note that $\la_x, \la_y$ depend on $\e$, but we do not use the subscript $\e$ in order to simplify the notation. We have
\begin{multline*}
g_B(\la_x,\la_y) 
\geq   \sup_{s\in[0,1]} p_B\left(z_x(s)-m(t_x(s))-z_y(s)+m(t_y(s)) \right) \vee \left| t_x(s)-t_y(s)\right| \\
-  \sup_{s\in[0,1]} \max  p_B(m(t_x(s))-m(t_y(s)))\vee \left| t_x(s)-t_y(s)\right|.
\end{multline*}
Since the map $t\to m(t) \in \S'_\R$ is continuous, observe that 
$$
\la'_v : \ s \to \left(z_v(s)-m(t_v(s)), t_v(s) \right), \quad v\equiv x,y
$$
is a parametric representation of the graph
$$
\g'_v := \left\{ (w,t) \in \S'_\R\times [0,T] : w\in\left[ v(t^-)-m(t) , v(t)-m(t)   \right]   \right\}, \quad v\equiv x,y.
$$
Consequently
\begin{multline}\label{lemmaA1-estim2}
\ddd_{{B},{\rm M1}}\left(\Psi_m(x), \Psi_m(y) \right) \leq g_B(\la_x,\la_y) + \sup_{s\in[0,1]} p_B(m(t_x(s))-m(t_y(s))) \vee \left| t_x(s)-t_y(s)\right| \\
\leq 2 \e +  \sup_{s\in[0,1]} p_B(m(t_x(s))-m(t_y(s))).
\end{multline}
Hence, by using the estimation of Step 1, we infer that:
\begin{equation}\label{lemmaA1-estim3}
\ddd_{{B},{\rm M1}}\left(\Psi_m(x), \Psi_m(y) \right) \leq C(B) \omega_m(\e),
\end{equation}
which in turn implies that $\Psi_m$ is continuous.

\emph{Step 3 (continuity of $\Psi_{\bf q}^n$):} 
Let us fix $n\geq 1$. Note that ${\bf q}_{u,m}^n$ maps $[0,T]$ into $\S_\R$, and the following holds:
\begin{equation}\label{th41-estim0}
\sup_{t\in[0,T]}\sup_{x \in \R}  \left| x^\a \p_x^\b {\bf q}_{u,m}^n(t,x)  \right| \leq C(L,\a) n^\b  \int_\R \left| \p_x^\b \xi(y) \right| \dd y, \quad \forall \a,\b \in \N.
\end{equation}
Owing to \eqref{th41-estim0}, we have ${\bf q}_{u,m}^n([0,T]) \subset B_n$, where $B_n$ is a bounded subset of $\S_\R$.  Let $\e>0$, $x,y \in \A$, and $\la_{x}:=(z_x,t_x), \la_y:=(z_y,t_y)$ be a  parametric representations of the graphs of $x$ and $y$ respectively such that
$$
g_{B_n}(\la_x,\la_y) \leq \e.
$$
We have 
\begin{multline*}
g_{B_n}(\la_x,\la_y) \geq  \sup_{s\in[0,1]}  \left| \int_0^L {\bf q}_{u,m}^n(t_x(s),.) \dd(z_x(s)-z_y(s)) \right|   \vee \left| t_x(s)-t_y(s)\right|  \\
\geq  \sup_{s\in[0,1]}  \left| \int_0^L {\bf q}_{u,m}^n(t_x(s),.) \dd z_x(s)-   \int_0^L {\bf q}_{u,m}^n(t_y(s),.) \dd z_y(s) \right|   \vee \left| t_x(s)-t_y(s)\right|  \\
-  \sup_{s\in[0,1]} \left|\int_0^L \left( {\bf q}_{u,m}^n(t_x(s),.)- {\bf q}_{u,m}^n(t_y(s),.) \right) \dd z_y(s)  \right|   \vee \left| t_x(s)-t_y(s)\right|.
\end{multline*}
Thus, it holds that
\begin{multline*}
\sup_{s\in[0,1]}  \left| \int_0^L {\bf q}_{u,m}^n(t_x(s),.) \dd z_x(s)-   \int_0^L {\bf q}_{u,m}^n(t_y(s),.) \dd z_y(s) \right|   \vee \left| t_x(s)-t_y(s)\right|  \\
\leq 2\e + \sup_{s\in[0,1]}\left| \int_0^L \left( {\bf q}_{u,m}^n(t_x(s),.)- {\bf q}_{u,m}^n(t_y(s),.) \right) \dd z_y(s)   \right|
\leq 2\e + \omega_2^n(\e).
\end{multline*}
where $\omega_2^n$ is the continuity modulus of ${\bf q}_{u,m}^n$. By noting that 
$$
\la''_v : \ s \to \left(\int_0^L {\bf q}_{u,m}^n(t_v(s),.) \dd z_v(s), t_v(s) \right), \quad v\equiv x,y
$$
is a parametric representation of the graph
$$
\g''_v := \left\{ (w,t) \in \S'_\R\times [0,T] : w\in\left[ \int_0^L {\bf q}_{u,m}^n(t,.) \dd v(t^-) , \int_0^L {\bf q}_{u,m}^n(t,.) \dd v(t)   \right]   \right\}, \quad v\equiv x,y,
$$
we deduce that
$$
\ddd_{{\rm M1}}\left(\Psi_{\bf q}^n(x), \Psi_{\bf q}^n(y) \right) \leq 2\e + \omega_2^n(\e).
$$
The proof is complete.
\end{proof}

\section{On uniqueness for solutions of Fokker-Planck equations} \label{App2}

In this part, we show that problem \eqref{eq:fp}-\eqref{eq:fp:bc} admits at most one weak solution in a wide class of positive Radon measures. 
We believe that this result is well-known, and we explain the proof for lack of precise reference.

Let us start by generalizing the notion of weak solution that is given in \eqref{eq:weak fp}. For any $m_0 \in  \P(\bar Q)$, we define a \emph{measure-valued weak solution} to \eqref{eq:fp}-\eqref{eq:fp:bc} to be a measure ${\bf m}$ on $\overline{Q}_T$ of the type  
$$
\dd {\bf m} = \dd m(t) \dd t,
$$ 
with $m(t) \in \tilde \P(\bar Q)$ for all $t\in [0,T]$, and $t \to m(t,A)$ measurable on $[0,T]$ for any Borel set $A \subset \bar Q$;
such that 
$$
\| b \|_{L_{\bf m}^2}^2:= \int_0^T \int_0^L |b|^2 \dd {\bf m} < \infty
$$ 
and
\begin{equation}
\label{eq:weak fp:measure}
\int_0^T \int_0^L (-\phi_t - \frac{\sigma^2}{2} \phi_{xx} + b \phi_x) \dd {\bf m} = \int_0^L \phi(0,.)\dd m_0
\end{equation}
for every $\phi \in \C^{test}$. We claim that such a solution is unique:

\begin{proposition}\label{uniqueness:measure-valued}
There is at most one measure-valued weak solution to \eqref{eq:fp}-\eqref{eq:fp:bc}.
\end{proposition}
\begin{proof}
Our approach is similar to \cite[Section 3.1]{porretta2015weak}. Let ${\bf m}$ be a measure-valued weak solution to \eqref{eq:fp}-\eqref{eq:fp:bc},
and consider the following dual problem:
\begin{equation}
\label{dual-problem:uniqueness} 
\left\{
\begin{aligned}
&-w_{t}-\frac{\sigma^{2}}{2} w_{xx} + {\bf b} w_x=\psi
 \quad \mbox{ in }Q_T,\\
& w(t,0) =w_{x}(t,L)=0\quad \mbox{ in } (0,T), \\
&w(T,x)=0 \quad \mbox{ in } Q,
\end{aligned}
\right.
\end{equation} 
where $\psi, {\bf b} \in \C^{\infty}(\overline{Q}_T)$. Let $w$ be a smooth solution to \eqref{dual-problem:uniqueness}. Since $w^2$ is smooth, we have:
$$
\int_0^T \int_0^L \left\{ -(w^2)_t - \frac{\sigma^2}{2} (w^2)_{xx} + b (w^2)_x \right\} \dd {\bf m} = \int_0^L w^2(0,.)\dd m_0.
$$
By \eqref{dual-problem:uniqueness} we thus have
$$
\int_0^T \int_0^L w(\psi-{\bf b} w_x)    \dd {\bf m} - \frac{\sigma^2}{2}\int_0^T \int_0^L |w_x|^2    \dd {\bf m} 
+  \frac{\sigma^2}{2}\int_0^T \int_0^L b w w_x    \dd {\bf m} = \int_0^L w^2(0,.)\dd m_0,
$$
so that 
$$
\frac{\sigma^2}{4}\int_0^T \int_0^L |w_x|^2    \dd {\bf m} \leq C \left( \| w \|_{\infty}^2  \int_0^T \int_0^L |b-{\bf b}|^2    \dd {\bf m}  + \| \psi\|_{\infty}\| w \|_{\infty} \right).
$$
Hence, from the maximum principle:
\begin{equation}\label{energy:estimate-uniqueness}
\int_0^T \int_0^L |w_x|^2    \dd {\bf m} \leq C \| \psi \|_{\infty}^2 \left(1+\| b-{\bf b} \|_{L_{\bf m}^2}^2 \right).
\end{equation}

Now, let ${\bf m}_1, {\bf m}_2$ be two measure-valued weak solutions to \eqref{eq:fp}-\eqref{eq:fp:bc}. We know that 
$$
b \in L_{{\bf m}_1}^2(Q_T) \cap L_{{\bf m}_2}^2(Q_T).
$$
Thus, $b \in L_{{\bf m}}^2(Q_T)$, where ${\bf m}= {\bf m}_1+{\bf m}_2$. Let $b^\e$ be a sequence of smooth functions converging to $b$ in $L_{{\bf m}}^2(Q_T)$. Since ${\bf m}$ is regular, note that such a sequence exists by density of smooth functions in $L_{{\bf m}}^2(Q_T)$. The measures ${\bf m}_1, {\bf m}_2$ being positive, $b^\e$ converges toward $b$ in $L_{{\bf m}_1}^2(Q_T) \cap L_{{\bf m}_2}^2(Q_T)$ as well. Now, let us consider $w^\e$ to be a solution to the dual problem that is obtained by replacing ${\bf b}$ by $b^\e$ in \eqref{dual-problem:uniqueness}.
By using $w^\e$ as a test function, we obtain
\begin{equation}
\label{eq:weak fp:measure}
\int_0^T \int_0^L \psi \dd( {\bf m}_1 - {\bf m}_2) =  \int_0^T \int_0^L (b-b^\e) w_x^\e \dd {\bf m}_2  -\int_0^T \int_0^L (b-b^\e) w_x^\e \dd {\bf m}_1  =: I_2^\e - I_1^\e.
\end{equation}
By virtue of \eqref{energy:estimate-uniqueness}, we have for $j=1,2$:
$$
\| w_x^\e \|_{L_{{\bf m}_j}^2} \leq C\| \psi \|_\infty \left(1+\| b-b^\e \|_{L_{{\bf m}_j}^2} \right) \leq C,
$$
so that
$$
\left| I_j^\e \right| \leq \| w_x^\e \|_{L_{{\bf m}_j}^2} \| b-b^\e \|_{L_{{\bf m}_j}^2} \leq C \| b-b^\e \|_{L_{{\bf m}_j}^2} \to 0, \quad \mbox{ as } \e \to 0.
$$
Consequently, for any smooth function $\psi$
$$
\int_0^T \int_0^L \psi \dd( {\bf m}_1 - {\bf m}_2) =0,
$$
which entails ${\bf m}_1 \equiv {\bf m}_2$ and concludes the proof.
\end{proof}


\end{document}